\documentclass[11pt]{amsart}
\usepackage{tikz, color}
\usepackage{musicography}
\usepackage[utf8]{inputenc}
\usepackage{graphicx}
\usepackage{amscd}
\usepackage{amsmath}
\usepackage{amsxtra}
\usepackage{amsfonts}
\usepackage{amssymb}
\usepackage{xcolor}
\usepackage{listings}
\usepackage{epigraph}
\newtheorem{theorem}{Theorem}[section]
\newtheorem{corollary}[theorem]{Corollary}

\newtheorem{proposition}[theorem]{Proposition}
\newtheorem{definition}[theorem]{Definition}
\newtheorem{remark}[theorem]{Remark}
\newtheorem{example}[theorem]{Example}
\newtheorem*{acknowledgements}{Acknowledgements}

\numberwithin{equation}{section}

\begin{document}

\title{Musical Systems with $\mathbb{Z}_n$ - Cayley Graphs }
\author{Gabriel Picioroaga}
\address{
[Gabriel Picioroaga] University of South Dakota\\
Department of Mathematical Sciences\\
414 E. Clark Street\\
Vermillion, SD, 57069\\
U.S.A\\}\email{Gabriel.Picioroaga@usd.edu}

\author{Olivia Roberts}
\address{
[Olivia Roberts] University of South Dakota\\
Department of Mathematical Sciences\\
414 E. Clark Street\\
Vermillion, SD, 57069\\
U.S.A\\}\email{Olivia.K.Roberts@coyotes.usd.edu}

\thanks{}
\subjclass[2010]{00A65, 13C99, 05A99}

\keywords{Cayley graph, finitely generated group, chords, scales, circle of fifths, counterpoint}

\begin{abstract} We apply geometric group theory to study and interpret known concepts from Western music. We show that chords, the circle of fifths, scales and the first species of counterpoint are encoded in the Cayley graph of the group $\mathbb{Z}_{12}$, generated by $3$ and $4$. Using $\mathbb{Z}_{12}$ as a model, we extend the above music concepts to a particular class of groups $\mathbb{Z}_{n}$, which displays geometric and algebraic features similar to $\mathbb{Z}_{12}$. We identify a weaker form of counterpoint which, in particular leads to Fux's dichotomy in $\mathbb{Z}_{12}$, and to consonant sets in $\mathbb{Z}_n$. Using Maple software, we implement these new constructions and show how to experiment with them musically.

\end{abstract}

\maketitle

\tableofcontents

\section{Introduction}

In this paper we focus on notes, chords, scales and a few rules that govern harmony such as the circle of fifths and certain aspects of counterpoint. Our aim is to present a unified and self-contained approach to these musical concepts and constructions, from the perspective of the so-called Cayley graphs. 
There are many notions of distance that are used in music theory, for example see \cite{Ty} or \cite{Vi} and references therein. By taking the viewpoint of the Cayley graph distance, we are able to keep chords, circle of fifths, scales, and consonant/dissonant dichotomies under the same framework. Generalization then becomes a natural mathematical step, especially because we maintain the characteristics of the framework. The Cayley graphs can be thought of as geometric objects associated to abstract groups. The connection with Western music comes through the cyclic group $\mathbb{Z}_{12}$, and more generally groups $\mathbb{Z}_n$,  where $n$ is a product of two relatively prime numbers. Specifically, we show that the oriented Cayley graph generated by the elements $3$ and $4$ encodes most of the chords used in Western classical music, the circle of fifths and the major/minor scales. By dropping the arrows, thus working in the unoriented (undirected) Cayley graph, one explains the first species of counterpoint through the action of affine  transformations (as in \cite{Ma}). We follow up with these ideas to the case $n=p\cdot q$ with relatively prime integers $p$ and $q$. We define chords, scales, and the circle of fifths in this general setting, which we call {\it{musical system}} when coupled with equal temperament tuning. We show that a weak form of counterpoint is always available, and provide examples where it can be extended toward full-fledged counterpoint partition.    \\
The paper is structured as follows: in Section \ref{term} we describe the notation we use for most math and music symbols. Other terms are described later when needed. Section \ref{groups} contains group theory definitions and examples together with geometric properties of Cayley graphs. We present conditions  under which automorphisms are isometries (Theorem \ref{iso}), and provide concrete examples. For the purpose of formalizing first species counterpoint, we define affine transformations as left translations of group automorphisms. In Section \ref{westmusic} we show how the oriented Cayley graph of $\mathbb{Z}_{12}$, generated by $3$ and $4$ helps define chords (path walks on the graph), the circle of fifths, and scale construction. 
The unoriented graph, obtained from the symmetric set $\{3,4, 8, 9\}$ is used to understand first species counterpoint. The formalism in the definition of counterpoint is inspired and taken from \cite{AJM} and \cite{Ma}, as is Theorem \ref{cpt12}. We prove Theorem \ref{cpt12} with the help of the machinery from Section \ref{groups}. We use the term {\it{affine}} for transformations on a Cayley graph, and we need only a group structure for it (see Definition \ref{aff} iii), also \cite{Loh} and \cite{Cai}). 
Let us mention that in \cite{AJM}, the group of affine transformations over $\mathbb{Z}_n$ that preserve the Cayley metric, is denoted by $\overrightarrow{\text{GL}}(\mathbb{Z}_n)$, and takes into consideration the ring structure as well in order to capture the double ``life" of a $k\in\mathbb{Z}_n$ as a note and as an interval. Inspired by \cite{Mco}, the first species counterpoint is studied in detail in \cite{AAL}. In these works, a multitude of {\it{strong dichotomies}} (we call these {\it{counterpoint partitions}}) are found with respect to unique affine transformations  on the ring.  In essence, by considering  only the group structure, coupled with suitable  Cayley graphs we found that we can still recover strong dichotomies, although not all of those from the $6$ affine equivalent classes in \cite{AJM}. Such reduction may be seen as a drawback, however it  better singles out the Fux dichotomy as a minimizer of Cayley graph distances. Another advantage is that it allows for a weak form of counterpoint partitions to be extended in case when $n$ is odd (a strong dichotomy existence implies $n$ even). Our process eliminates many dichotomies because of the following requirement, which is convenient to consider for Western music: the symmetric set of the group's generators must be consonant. Then we ask compatibility with an affine isometry $T$ (with respect to the Cayley graph) and idempotent, i.e. $T^2=Id$ (properties we call the {\it{weak counterpoint condition}}). These lead us to isolate the Fux dichotomy among four possible counterpoint partitions  (see Theorem \ref{4p}). We also observe that minimizing further on the path lengths of the oriented Cayley graph, the only possible counterpoint partition is Fux's. Theorem \ref{last} in Section \ref{systems} shows that under mild restrictions a weak from of counterpoint is available in our setting. Based on it we show a few examples which lead to counterpoint partitions ($n=10$ and $n=12$) or maximal consonant sets (n=15). These examples suggest that for $n=pq$ and gcd$\{p,q\}=1$, the weak counterpoint condition implies the existence of a maximal consonant/dissonant pair of subsets of $\mathbb{Z}_n$, each of cardinal $[n/2]$. We will study such results in a future work. \\
In Section \ref{systems} we also define chords, scales and a generalized circle of fifths concept, inspired by the $\mathbb{Z}_{12}$ case, and compatible with the Cayley graph framework. Generalized circle of fifths (based on relatively primes which generate the cyclic group $\mathbb{Z}_n$) and scales (almost equidistant partitions of an octave) already occurred  in \cite{CM}). Although generalized chords were missing from \cite{CM}, later work supplied the addition, see for example \cite{Baz}, \cite{Gou} and \cite{AKK} and references therein. In \cite{AKK} for example, the connection between triads, the fifth (which occupies a prominent role in Clough and Myerson construction), and the Cayley graph of $\mathbb{Z}_{12}$ generated by $\{3,4,7\}$ is clearly spelled out. The general system in these studies takes the generalized fifth $2k+1$, the value $p:=k+1$ and the value $n:=4k$ (number of microtones) as the backbone of the microtonal system. While this choice clearly generalizes $\mathbb{Z}_{12}$ for $k=3$, it becomes incompatible with a factorization $n=pq$ unless $k=3$. It would be interesting to find if an analogue of first species counterpoint can be defined in this setting.  In our considerations, the generalized fifth comes out naturally after imposing $n$ be the product of two coprimes $p$ and $q$, its value being $p+q$. Hence, the unifying character of the generalization (chords, circles of fifths, scales, and first species counterpoint) is represented by the group generators $p$ and $q$ and the corresponding Cayley graph. 
As application we have written Maple code to experiment with the sound of chords, scales, the circle of fifths, and counterpoint in various musical systems $(\mathbb{Z}_n, s)$, where $s$ encodes the total frequency length (octave for $s=2$). The tuning we consider is the equal temperament one, i.e. an equidistant division of the frequency space. In Section \ref{sound} we describe how to implement it, together with notes and chords in Maple. Our Maple code is available to the interested reader by email request.

\section{Terminology and Notation}\label{term} 
\noindent We describe some basic symbols and notation from math and music used throughout the paper. Most of the time we use capital letters to denote sets except the empty set which is denoted by $\emptyset$. If $A$ and $B$ are two sets, then by $A\setminus B$ we denote the difference set. For a set $A$ the number of elements in $A$ is denoted $|A|$. By $\mathbb{N}$ we denote the set of positive integers. The integers are denoted by $\mathbb{Z}$, and in music it represents the discrete pitch class space. If $n\in\mathbb{N}$ and $n\geq 2$, the finite set $\mathbb{Z}_n=\{0,1,..., n-1\}$ represents both the set of remainders obtained by division by $n$ and the class of integers modulo $n$. For $a$ and $n\geq 2$ in $\mathbb{N}$ the number $a\text{ mod }n \in \mathbb{Z}_n$ denotes the reminder of the division of $a$ by $n$.
The greatest common divisor of $p\in\mathbb{Z}$ and $q\in\mathbb{Z}$ is denoted by gcd$\{p,q\}$. The elements of $\mathbb{Z}_{12}$ are put in one-to-one correspondence with the symbols $C$, $C \sh /D \fl$, $D$, $D\sh /E \fl$, $E$, $F$, $F \sh /G \fl$, $G$, $G\sh /A \fl$, $A$, $A \sh /B \fl$, $B$ that form the notes of the chromatic scale. 
\vspace{0.2in}
\begin{center}
\begin{tabular}{ |c|c|c|c|c|c|c|c|c|c|c|c | } 
 \hline
 0 & 1 & 2 & 3 & 4 & 5 & 6 & 7  & 8 & 9 & 10 & 11 \\ 
 \hline
$C$ & $C \sh /D \fl$ & $D$ & $D\sh /E \fl$ &  $E$ & $F$ & $F \sh /G \fl$ & $G$ & $G\sh /A \fl$ & $A$ & $A \sh /B \fl$ & $B$ \\
 \hline
\end{tabular}
\end{center}
\vspace{0.2in}
To avoid confusions with classic musical interval definitions (third, fifth etc.), we mention that we also denote as $\mathbb{Z}_{12}$ the set of intervals, and more generally by $\mathbb{Z}_n$. Any peril of confusion with pitch classes is eliminated by the context we will work in. Music intervals and their associated frequency ratios will be explained algebraically in the later section related to tuning. We caution the reader about the concept ``distance". Above, it is used in the usual sense as distance between two real numbers (or ``counting distance" when referred to integers). We will make sure to distinguish it from ``distance" between two vertices on a graph, even though the vertices are represented by integers. 

\section{Groups, Generators and Cayley Graphs}\label{groups}

\noindent In this section, we mention some results from group theory that we need in the sequel. For an elementary introduction in group theory we refer to \cite{Ga} and \cite{TJ}.

\begin{definition} Let $(G, \star)$ be a group with its identity element denoted by $e$. A finite set $S:=\{g_1, g_2, ..., g_n\}\subset G$ is a set of generators if $g_i\neq e$, for all $i$, and for any $g \in G$ there exists $g_{i_1}, \text{ } g_{i_2} , ... , \text{ } g_{i_k} \in S$ such that $g=g_{i_1}^{\varepsilon_1} \star g_{i_2}^{\varepsilon_2} \star ... \star g_{i_k}^{\varepsilon_k}$ where $\varepsilon_j = \pm1$. 
We say that $G$ is finitely generated and write $G=\langle S, S^{-1}\rangle$, where $S^{-1}= \{g_1^{-1},g_2^{-1},...,g_n^{-1} \}$. A set of generators is minimal if when removing an element form it, the resulting set does not generate $G$.
\end{definition}
\begin{remark}
A set $S$ generates $G$ if any element of $G$ can be written as a word over the ``alphabet" $S\cup S^{-1}$. One can rewrite a word in its reduced form, i.e. in the definition above, if some  $i_j=i_{j+1}$ then $\varepsilon_j = \varepsilon_{j+1}$.      
\end{remark}

\noindent Group theory also deals with infinitely generated groups, however we do not need it in our considerations. Notice that a group can be infinite as a set and be finitely generated. Any element of the group can be written as a finite product of some of the generators. We will see below examples of groups that admit different sets of generators. Let us highlight the following property: in a finite group one can find a set of generators $S$ such that for every $g\in S$ either $g^{-1}\notin S$ or $g=g^{-1}$, i.e. $S$ consists of 'positive' generators. This follows from the fact that in a finite group every element must have a finite order, e.g. see \cite{Ga}. 

\begin{example} Let $n\in\mathbb{N}$, $n\geq 2$. The pair $(\mathbb{Z}_n, \oplus)$ forms an abelian group with $a\oplus b:= (a+b)\text{ mod }n$ for $a$, $b$ in $\mathbb{Z}_n$. The inverse of $a$ is $\ominus a := n-a$. We will use the notation $a\ominus b$ for $a\oplus (n-b)$. The following sets generate $\mathbb{Z}_n$:  
$S=\{1\}$ and 
$S=\{ k\}$ for $k\in\mathbb{Z}_n$ with gcd$\{k,n \}=1$ (i.e. $\mathbb{Z}_n$ is cyclic). There can be more diverse sets of generators; if $n$ is a product $p_1p_2...p_k$ with gcd$\{p_i, p_j \}=1$ for all $i\neq j$, then the set $S:=\{p_1, p_2, ..., p_k \}$ generates $\mathbb{Z}_n$. 
\end{example}
\begin{example}
Similar to mod $n$ addition on integers, there is the mod $n$ multiplication. For $a, b\in\mathbb{Z}_n$, define 
  $a\odot b:= (ab)\text{ mod }n$. The distributive properties of $\odot$ with respect to $\oplus$ work similarly as in the case of integer multiplication and addition. Also $a\odot b\oplus c=ab\oplus c=(ab+c)\text{ mod }n$ for all $a,b,c\in\mathbb{Z}_n$. The following set $U(n):=\{ k\in\mathbb{Z}_n\text{ }|\text{ } \text{ gcd}(k,n)=1 \}$ is called the set of units of $\mathbb{Z}_n$, and plays a crucial role in the sequel. 
While $\odot$ is a binary operation in $\mathbb{Z}_n$, the invertibility property may be violated. However, restricting  $\odot$ to $U(n)$ will do; the pair $(U(n), \odot)$ is an abelian group.      
\end{example}
\noindent From here on, in an arbitrary group $G$, the binary operation will be denoted using multiplicative notation, i.e. instead of $g\star h$ we write $gh$. In particular cases, we will use the classic established notations without peril of confusion.    
\begin{definition} Let $G$ be a group generated by a set $S$ of positive generators (i.e. if $w\in S$ then $w^{-1}\notin S$ unless $w=w^{-1}$). The oriented Cayley graph of $G$ associated to this set of generators is defined as the pair $(V,E)$ where $V=\{g \text{ }| \text{ }g\in G\}$ is the set of vertices and $E \subset V \times V$ is the set of edges: $(g,h) \in E$ if and only if there exists $w\in S$ such that $h = g w$. The edge $(g,h)$ is labelled $w$, has source vertex $g$ and points toward vertex $h$.
\end{definition}
\noindent We will use arrows to show the edge $(g,h)$ as a path $g\to h$ on the oriented graph. If we remove the arrows, the graph obtained will be called unoriented Cayley graph. In some textbooks, e.g. \cite{Ga}, the oriented Cayley graph from above definition is called Cayley digraph. Prior to dropping arrows, one can define the unoriented Cayley graph concept similarly, assuming the generating set $S$ is symmetric, i.e. $S=S^{-1}$. We do not need the symmetry condition when the oriented Cayley graph is considered, owing to the property on finite groups mentioned above (and we can pass to a minimal generating set). The reason we need both concepts is that we will ``walk" on minimal paths in both (un)oriented graphs, and the result may not be the same. 

\begin{example} From the example above, we can write $Z_6=\langle 2,3 \rangle$ and $Z_{12}=\langle 3,4 \rangle$. Their corresponding oriented Cayley graphs are depicted in Figures 1 and 2. We know that $\mathbb{Z}_n$ is cyclic, hence one can draw a very simple oriented Cayley graph with $n$ vertices placed on a circle with consecutive arrows (loop). Also, $Z_6=\langle 4,3 \rangle$, i.e. the inverses of $2$ and $3$ generate $\mathbb{Z}_6$. For these generators, the Cayley graph is similar with the one in Figure 1,  except the horizontal edges where the arrows are reversed and labeled $4$. One can treat $Z_{12}=\langle 9,8\rangle =\langle 3,8 \rangle= \langle 4,9\rangle $ similarly to obtain different oriented Cayley graphs.   
\end{example}

\begin{definition}\label{d} 
i) A path (between) from $g$ to (and) $h$ of length $k$ in the (un)oriented Cayley graph is a set of vertices $\{x_0,x_1,...,x_{k-1}, x_{k} \}$ such that $x_0=g$, $x_{k}=h$ and $\forall$ $i=0,..,k-1$ the pair  $(x_i, x_{i+1})$ is an edge, i.e. $\exists$ $w_i$ generator such that $x_{i+1}=x_i w_i$. If the oriented version is considered we write $x_0\to x_1\to...\to x_{k} $.\\
ii) The function $d: G\times G\to [0,\infty)$ defined by $d(g,h)=0$ if $g=h$ and
$$
d(g,h):=\text{ min }\{ k\in\mathbb{N} \text{ }|\text{  there exists an unoriented path of length }k\text{ from }g\text{ to }h   \}\text{ if }g\neq h
$$
is called distance.
\end{definition}
\begin{remark} The distance function is a metric on the unoriented Cayley graph, i.e. $d$ is symmetric, $d(g,h)=d(h,g)$, non-negative, and satisfies the triangle inequality, $d(g,h)\leq d(g, w)+d(w,h)$, $\forall$ $g,h,w \in G$. The metric properties do not hold if the minimum above is taken over oriented paths only. For example, even the symmetry property is broken; in Figure 1 one can see that from vertex $1$ to $3$ we have an oriented path of length $1$, whereas if we measure from $3$ to $1$ the length is $2$. To eliminate any confusion, wherever we use ``distance" between vertices, then the unoriented Cayley graph with its metric are used.    
\end{remark}

\tikzset{
    to/.style={
        ->,
        thick,
        shorten <= 1pt,
        shorten >= 1pt,},
    from/.style={
        <-,
        thick,
        shorten <= 1pt,
        shorten >= 1pt,}
}

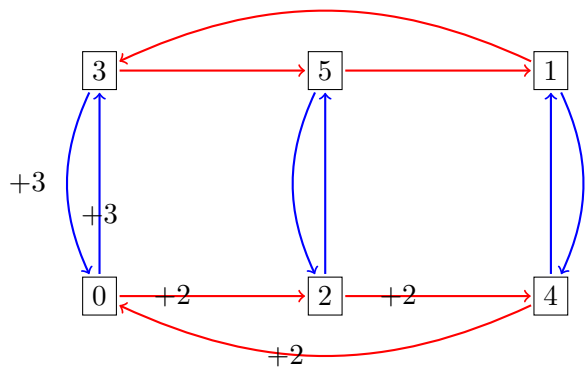
\begin{figure}\label{zsix}
  \begin{center}
      \begin{tikzpicture}[scale=1.5]
           
            \node[shape=rectangle, draw=black] (0) at (0,0) {$0$};
            \node[shape=rectangle, draw=black] (2) at (2,0) {$2$};
            \node[shape=rectangle, draw=black] (4) at (4,0) {$4$};

            \node[shape=rectangle, draw=black] (3) at (0,2) {$3$};
            
            \node[shape=rectangle, draw=black] (5) at (2,2) {$5$};
            \node[shape=rectangle, draw=black] (1) at (4,2) {$1$}; 

            \path (0) edge[to, draw=blue] node[label=below:{+3}] {} (3);
            \path (2) edge[to, draw=blue] node {} (5);
            \path (4) edge[to, draw=blue] node {} (1);

            \path (0) edge[to, draw=red] node[label=left:{+2}] {} (2);
            \path (2) edge[to, draw=red] node[label=left:{+2}] {} (4);
            \path (3) edge[to, draw=red] node {} (5);
            \path (5) edge[to, draw=red] node {} (1);
            
            \path (1) edge[to, bend right=25, draw=red] node {} (3);
            \path (4) edge[to, bend left=25, draw=red] node[label=left:{+2}] {} (0);
            \path (1) edge[to, bend left=25, draw=blue] node {} (4);
            \path (5) edge[to, bend right=25, draw=blue] node {} (2);
            \path (3) edge[to, bend right=25, draw=blue] node[label=left:{+3}] {} (0);

        \end{tikzpicture}
        \caption{ $\mathbb{Z}_{6}= \langle 2, 3\rangle$  oriented Cayley graph }
  \end{center}
\end{figure}

   \begin{figure}\label{ztw}  
 \begin{center}
        \begin{tikzpicture}[scale=1]

            \node[shape=rectangle, draw=black] (0) at (0,0) {$0$};
            \node[shape=rectangle, draw=black] (4) at (0,2) {$4$};
            \node[shape=rectangle, draw=black] (8) at (0,4) {$8$};

            \node[shape=rectangle, draw=black] (3) at (2,0) {$3$};
            \node[shape=rectangle, draw=black] (6) at (4,0) {$6$};
            \node[shape=rectangle, draw=black] (9) at (6,0) {$9$};
            
            \node[shape=rectangle, draw=black] (7) at (2,2) {$7$};
            \node[shape=rectangle, draw=black] (11) at (2,4) {$11$}; 
            \node[shape=rectangle, draw=black] (10) at (4,2) {$10$}; 
            \node[shape=rectangle, draw=black] (1) at (6,2) {$1$}; 
            \node[shape=rectangle, draw=black] (2) at (4,4) {$2$}; 
            \node[shape=rectangle, draw=black] (5) at (6,4) {$5$}; 
            
            \path (0) edge[to, draw=blue] node[label=below:{+4}] {} (4);
            \path (4) edge[to, draw=blue] node[label=below:{+4}] {} (8);
            \path (8) edge[to, bend right=25, draw=blue] node[label=below:{+4}] {} (0);
            \path (3) edge[to, draw=blue] node {} (7);
            \path (7) edge[to, draw=blue] node {} (11);
            \path (6) edge[to, draw=blue] node {} (10);
            \path (10) edge[to, draw=blue] node {} (2);
            \path (9) edge[to, draw=blue] node {} (1);
            \path (1) edge[to, draw=blue] node {} (5);
            \path (11) edge[to, bend right=25, draw=blue] node {} (3);
            \path (2) edge[to, bend left=25, draw=blue] node {} (6);
            \path (5) edge[to, bend left=25, draw=blue] node {} (9);

            \path (0) edge[to, draw=red] node[label=left:{+3}] {} (3);
            \path (3) edge[to, draw=red] node[label=left:{+3}] {} (6);
            \path (6) edge[to, draw=red] node[label=left:{+3}] {} (9);
            \path (9) edge[to, bend left=25, draw=red] node[label=left:{+3}] {} (0); 
            \path (4) edge[to, draw=red] node {} (7);
            \path (7) edge[to, draw=red] node {} (10);
            \path (10) edge[to, draw=red] node {} (1);
            \path (8) edge[to, draw=red] node {} (11);
            \path (11) edge[to, draw=red] node {} (2);
            \path (2) edge[to, draw=red] node {} (5);
            \path (5) edge[to, bend right=25, draw=red] node {} (8);
            \path (1) edge[to, bend right=25, draw=red] node {} (4);

        \end{tikzpicture}
        \caption{ $\mathbb{Z}_{12}= \langle 3, 4\rangle$  oriented Cayley graph }
   \end{center}
\end{figure}
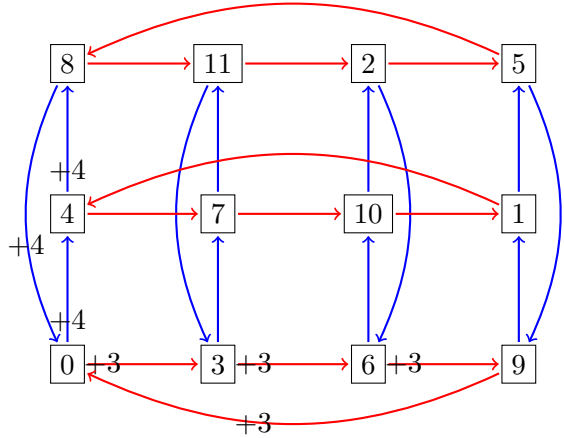

\noindent Let us highlight another important property of the metric $d$. If $w\in S\cup S^{-1}$ then $d(g,gw)=1$, $\forall$ $g\in G$. This entails that $d$ is left-translation invariant: $d(g, h)=d(\omega g, \omega h)$, $\forall$ $\omega \in G$. This holds because $g^{-1}h=(\omega g)^{-1} \omega h$. We will focus later on an interesting connection between certain functions that preserve the Cayley graph distance and the music composition technique counterpoint. For this reason we need to mention a few more group theory concepts and results. 
\begin{definition}
 Let $G$ and $H$ be a two groups. A function $f:G\to H$ is called (group) morphism if $f(g_1g_2)=f(g_1)f(g_2)$, $\forall$ $g_{1,2}\in G$. If moreover $f$ is bijective, then $f$ is called isomorphism, and we say that the groups $G$ and $H$ are isomorphic. An isomorphism $f: G\to G$ is called an automorphism of $G$.    
\end{definition}

\begin{definition}\label{aff} Let $G$ be a group generated by a finite, symmetric set, $S=S^{-1}$. \\
i) A function $f: G\to G$ is an isometry on the unoriented Cayley graph (of $G$ with respect to $S$) if the following identity holds 
\begin{equation}\label{e1}
d(x,y)=d(f(x), f(y)), \quad \forall x,y\in G
\end{equation}
where $d$ is the distance from Definition \ref{d} ii).\\
ii) A function $\varphi: G\to G$ is a right-translation if $\exists$ $w\in G$ such that
\begin{equation}\label{e2}
\varphi(g)= gw, \quad \forall g\in G
\end{equation}
Left-translation is defined similarly. In the case that $G$ is abelian, the two concepts coincide.\\
iii) A (left or right) translation of an automorphism of $G$ is called affine transformation. 
\end{definition}
\noindent We do not expand here on the graph isomorphism concept, though in essence this is what we obtain, indirectly in the next theorem  about a group automorphism which preserves the generating set. The  rigidity question, i.e. under what conditions is a graph isomorphism an affine transformation, is important in geometric group theory. Nevertheless, this topic is beyond our goals. For more details, the interested reader may consult e.g. \cite{Loh} and references therein such as \cite{Cai}. 
        
\begin{remark}
  The set of all automorphisms of a group $G$ is denoted by Aut$(G)$. This set is also a group; the binary operation is the usual function composition $f_1\circ f_2 (g) = f_1(f_2(g))$, $\forall$ $f_{1,2}\in \text{Aut}(G)$, $\forall$ $g\in G$. Its identity element is the identity function $\text{Id}:G\to G$, $Id(g)=g$, $\forall$ $g\in G$; the inverse of $f$ is the usual inverse function $f^{-1}$, which exists because $f$ is bijective. One can check easily that any (left or right) translation is an isometry. Clearly, a translation is not a group morphism, unless $w=e$ in (\ref{e2}). Also, an isometry need not be a morphism, and vice-versa. We will provide below a sufficient condition for an automorphism to be an isometry.
\end{remark}

\begin{example}\label{aut} 
For $n=12$, $U(12)=\{ 1,5,7,11\}$. We do not need it here but one can check $\{ 5,7 \}$ is a minimal generating set for $U(12)$. More importantly, we need the following known Theorem (see \cite{Ga}): The groups Aut$(\mathbb{Z}_n)$ and $U(n)$ are isomorphic. As a consequence, Aut$(\mathbb{Z}_{12})= \{ f_1, f_2, f_3, f_4\} $ where $f_1(g)=Id (g)=g$, $\forall$ $g\in \mathbb{Z}_{12}$,
$f_2(g) = 5\odot g$, $\forall$ $g\in \mathbb{Z}_{12}$, 
$f_3(g) = 7\odot g$, $\forall$ $g\in \mathbb{Z}_{12}$,
$f_4(g) = 11\odot g$, $\forall$ $g\in \mathbb{Z}_{12}$. Note that for each $i=1,2,3,4$ we have $f_i^2=f_i\circ f_i=Id$. It is not obvious, though one may check by a tedious calculation, that each $f_i$ is an isometry on the unoriented Cayley graph of $\mathbb{Z}_{12}=\langle 3,4,8,9\rangle $. 
  \end{example}

\noindent The next result provides a simple criteria for checking whether a group automorphism is a Cayley graph isometry. The statement can be viewed as the easier counterpart of the rigidity problem mentioned above, and is considered somehow implicit in the literature of Cayley isomorphism graphs and groups, see e.g. \cite{Cai}. Because we will use it extensively for elements in Aut($\mathbb{Z}_n$), we formulate it as a theorem with complete proof.  

\begin{theorem}\label{iso} Let $G$ be a group, $e$ its identity element, and $f\in\text{Aut}(G)$. Suppose $S\subset G$ is a symmetric, generating set for $G$. The following are equivalent: 
\\
i) $f(S)=S$.\\
ii) $f$ is an isometry on the unoriented Cayley graph of $G$ with respect to $S$.
\end{theorem} 
\begin{proof} 
i)$\Longrightarrow$ii)
For any  $x,y$ elements of $G$ we have $d(x,y)\in\mathbb{N}$.  We prove (\ref{e1}) by induction over $m:=d(x,y)$.  
If $m=1$ then $y=xw$ for some $w\in S$. Hence $x^{-1}y\in S$. By hypothesis, $f(x^{-1}y)\in S$. Because $f$ is a morphism, $f(x^{-1}y)=f(x)^{-1}f(y)$, therefore $f(x)^{-1}f(y)\in S$. This implies $d(e, f(x)^{-1}f(y)) =1 $. By left-invariance and morphism properties  $1=d(e, f(x)^{-1}f(y)) =  d(f(x), f(y))$, and the first step in induction is verified. We assume now (\ref{e1}) holds for a fixed $m\in\mathbb{N}$ and any $x, y $ with $d(x,y)=m$. Let $x', y'$ in $G$ with $d(x', y')=m+1$. We  want to show $d(f(x'), f(y'))=m+1$. Let $x':=x_0, x_1, ..., x_m, x_{m+1}=y'$ be a path of length $m+1$ in the unoriented Cayley graph with respect to $S$. Because this path realizes the minimum length between $x'$ and $y'$, we have that $d(x', x_m)=m$. By the induction hypothesis $d(f(x'), f(x_m))=m$. This implies that $f(x'), f(x_1),..., f(x_m)$ is a path of length precisely $m$, because $f$ is injective. From the path definition we know $x_m^{-1}y'\in S$, hence $f(x_m)^{-1}f(y')\in S$. If follows that $f(x'), f(x_1),..., f(x_m), f(y')$ must be a path of length $m+1$, hence $d(f(x'), f(y'))\leq m+1$, by the min condition on $d$. If, by contradiction $d(f(x'), f(y'))< m+1$ then there would be a path $f(x')=z_0, z_1,..., z_t=f(y')$ with $t<m+1$. Because $f$ is an automorphism, there exist $y_0=x',y_1,..., y_{t-1}, y_{t}=y'$ in $G$ such that  $f(y_i)=z_i$ for all $i=0,...,t$. The path condition implies  $f(y_{i+1})=f(y_i)s_i$ for some $s_i\in S$. Using $f(S)=S$ we have $\forall$ $i=0,..,t$ $\exists$ $w_i\in S$ such that $f(s_i)=w_i$. Then $f(y_{i+1})=f(y_i)f(w_i)=f(y_iw_i)$, which entails (because $f$ is injective) $y_{i+1}=y_i w_i$ for all $i=0,...,t$. Hence we would obtain that $x'=y_0,y_1,..., y_{t-1}, y_t=y'$ is a path between $x'$ and $y'$. Then by min condition in the definition of $d$, it would follow that $d(x', y')\leq t$ which contradicts $t<m+1$. In conclusion $d(f(x'), f(y'))=m+1$, and the induction step is completed. 
\\
ii)$\Longrightarrow$i) Let $s\in S$ arbitrary. Then $d(f(s), f(e))=d(s,e)=1$, hence $d(e, f(s))=1$. By the definition of $d$ it follows that $\exists$ $w\in S$ such that $f(s)=e w=w$, hence $f(s)\in S$. We obtain $f(S)\subseteq S$. Because $f$ is automorphism and preserves $d$, it follows that $f^{-1}$ is automorphism and preserves $d$. From the first part of the implication, now applied to $f^{-1}$,  we get $f^{-1}(S)\subseteq S$. This implies $S\subseteq f(S)$. In conclusion $f(S)=S$. 
\end{proof}
\begin{example}\label{isom}
  Let $\{f_i, i=1,2,3,4\}$ be the automorphisms from Example \ref{aut}. For $S$ the symmetric set $\{3,4,8,9\}$ we check $f_i(S)=S$ for all $i=1,2,3,4$. Obviously $f_1(S)=S$. We have:\\
 $\bullet$  $f_2(3)=3$, $f_2(4)=8$, $f_2(8)=4$, $f_2(9)=9$, hence $f_2(S)=S$.\\
 $\bullet$ $f_3(3)=9$, $f_3(4)=4$, $f_3(8)=8$, $f_2(9)=3$, hence $f_3(S)=S$.\\
 $\bullet$ $f_4(3)=9$, $f_4(4)=8$, $f_4(8)=4$, $f_4(9)=3$, hence $f_4(S)=S$.\\
 By Theorem \ref{iso} all $f_i$ are isometries on the unoriented Cayley graph of $\mathbb{Z}_{12} $ with respect to $S$. 
 The next example shows that automorphisms need not be isometries (with respect to a Cayley graph). We take the group $\mathbb{Z}_{10}$ and the symmetric generating set $S=\{2,5,8\}$. Because gcd$\{3,10\}=1 $ the function $f:\mathbb{Z}_{10}\to \mathbb{Z}_{10} $, $f(x)=3\odot x= 3x\text{ mod }10$ is an automorphism (by the result mentioned in Example \ref{aut}). However, $1=d(0,2)\neq d(f(0), f(2))=d(0, 6)=2$ (the path $6,8,0$ between $6$ and $0$ gives the minimum length equal to $2$). Notice also that $2\in S$ and $f(2)=6 \notin S$, hence there exists automorphisms $f$ and symmetric sets $S$ such that $f(S)\neq S$.    
\end{example}
\noindent Because left translations preserve the Cayley metric, using the last theorem we obtain the corollary below. This result is used in the next section to highlight a mathematical (algebraic and geometric) feature of counterpoint. 
\begin{corollary}\label{cor} Let $f\in \text{Aut}(G)$ as in Theorem \ref{iso} i), and $w\in G$ fixed. Then the affine transformation  $L(g)=wf(g)$, $\forall$ $g\in G$, is an isometry on the Cayley graph of $G$ with respect to $S$. 
\end{corollary}
\noindent Certain affine transformations will help explain counterpoint later on. To that end the theorem below characterizes affine reflections. We will need this result only in  the $\mathbb{Z}_n$ setting, though the proof is straightforward in general for finitely generated groups.
\begin{proposition}\label{ccpt}
  Let $G$ be a group, $\varphi\in\text{Aut }(G)$, and $w\in G$. If $T$ is either of the affine transformations $R(g)=\varphi(g)w$, $\forall$ $g\in G$, or $L(g)=w\varphi(g)$,  $\forall$ $g\in G$, then the following are equivalent:\\
    i) $T^2=Id$;\\
    ii) $\varphi^2=Id$ and $\varphi(w) w=e$.\\
   In particular, if $G=\mathbb{Z}_n$ and for some $h\in U(n)$, $\varphi(g)=h\odot g$ $\forall$ $g\in\mathbb{Z}_n$, then i) and ii) are equivalent to 
   iii) $h^2g \oplus h w\oplus w=g$ for all $g\in\mathbb{Z}_n$.   
\end{proposition}
\begin{proof} In a group, the inverse exists and is unique, therefore $ab=e$ and $ba=e$ are equivalent for any $a,b$ in $G$. Thus $\varphi(w) w=e$ is equivalent to $w\varphi(w) =e$. It suffices then to prove i)$\iff$ii) e.g. for $T=R$ ( the argument is similar if $T=L$, to get condition $\varphi^2=Id$). We have $R(R(g))=g$, $\forall$ $g\in G$ $\iff$ $\varphi(\varphi(g)w)w=g $, $\forall$ $g\in G$ $\iff$ $\varphi(\varphi(g))\varphi(w)w=g $, $\forall$ $g\in G$. The last condition clearly holds if ii) is satisfied, hence we get i). If i) is satisfied then the last condition implies, on one hand  $\varphi(w)w=e$ (plugging in it $g=e$), and on another hand $\varphi(\varphi(g))=g$, $\forall$ $g\in G$ (by updating it). Hence $\varphi^2=Id$ and we get ii). Hence i) $\iff$ ii). Condition iii) is an update of ii) when $G=\mathbb{Z}_n$ and $\varphi\in \text{ Aut }(\mathbb{Z}_n)$.
 \end{proof}
\begin{example}\label{refl}
 There always exists $\varphi\in \text{ Aut }(G)$ such that $\varphi^2=Id$ and $\varphi$ is an isometry with respect to any generating, symmetric set $S\subset G$. The map $\varphi(g)=g^{-1}$, $\forall$ $g\in G$ easily satisfies $\varphi^2(g)=g$, $\forall$ $g$. Because $S=S^{-1}$, using Theorem \ref{iso}, $\varphi$ is an isometry on the unoriented Cayley graph.   
\end{example}

\section{Understanding Western Music with $\mathbb{Z}_{12}$}\label{westmusic}
\subsection{Chords} In $\mathbb{Z}_{12}$ terms, the $C$-major triad $C-E-G$ is encoded as the sequence  $0-4-7$. The $C$ minor chord is encoded as $0-3-7$. One can observe the following pattern: naming $0$ the root of the chord, add in succession $+4$, $+3$ (major). For the minor chord the order is reversed. We can infer therefore, the following definition where the paths are read in $\mathbb{Z}_{12}$'s graph, see Figure 2.
\begin{definition}\label{cho}
Let $x\in\mathbb{Z}_{12}$. An $x-$major triad chord is the path $x\to x+4\to x+4+3$. The $x-$minor triad chord is the path $x\to x+3\to x+3+4$.  
More generally, an $x-$major (minor) chord is a path $x_1\to x_2\to\cdots \to x_k$ in the oriented Cayley graph such that $x=x_1$ and:  \\
i) $x_{2}=x_{1}+4$ ( $x_{2}=x_{1}+3$ for minor )\\
ii) $x_{3}=x_{2}+3$ ( $x_{3}=x_{2}+4$ for minor )\\
iii) $x_{i+1}\ominus x_i\in\{3,4\}$, $\forall$ $i=1,...,k$ \\
iv) the path is non-self intersecting unless $x_k=x_1$.
\end{definition}
\noindent The first two conditions spell out the patterns $+4+3$ for the start of a major chord and $+3+4$ for the start of a minor one. The third condition allows the option of oscillation or repetition of generators $3$ and $4$ as counting distances between notes. In this paper, when we need a major (minor) chord then iii) is taken with alternating order $+4,+3$ ( $+3,+4$ for minor) all the way. The last condition restricts a chord from wandering on the graph.   

\begin{example} 
We list below a set of major and minor chords together with their classical music names. The patterns can be read on Figure 2. \\
\\
 \textbf{Triads} \\
 Major Triad: $\textcolor{blue}{+4},\textcolor{red}{+3}$
 \\ Minor Triad: $\textcolor{red}{+3},\textcolor{blue}{+4}$
 \\ Diminished Triad: $\textcolor{red}{+3},\textcolor{red}{+3}$
 \\ Augmented Triad: $\textcolor{blue}{+4},\textcolor{blue}{+4}$
 \\
 \\
 \textbf{ $7^\text{th}$ Chords}
 \\Major $7^\text{th}$ Chord: $\textcolor{blue}{+4},\textcolor{red}{+3},\textcolor{blue}{+4}$
 \\
 Dominant $7^\text{th}$ Chord: $\textcolor{blue}{+4},\textcolor{red}{+3},\textcolor{red}{+3}$
 \\Minor $7^\text{th}$ Chord: $\textcolor{red}{+3},\textcolor{blue}{+4},\textcolor{red}{+3}$
  \\Fully Diminished $7^\text{th}$ Chord: $\textcolor{red}{+3},\textcolor{red}{+3},\textcolor{red}{+3}$
 \\Half Diminished $7^\text{th}$ Chord: $\textcolor{red}{+3},\textcolor{red}{+3},\textcolor{blue}{+4}$
 \\Augmented Major $7^\text{th}$ Chord: $\textcolor{blue}{+4},\textcolor{blue}{+4},\textcolor{red}{+3}$
\\
 \\
 \textbf{ $9^\text{th}$ Chords}\\
Major: $\textcolor{blue}{+4},\textcolor{red}{+3},\textcolor{blue}{+4},\textcolor{red}{+3} $  \\
Minor: $\textcolor{red}{+3},\textcolor{blue}{+4},\textcolor{red}{+3},\textcolor{blue}{+4} $\\
Dominant 9: $\textcolor{blue}{+4},\textcolor{red}{+3},\textcolor{red}{+3},\textcolor{blue}{+4} $\\
Dominant Flat 9:  $\textcolor{blue}{+4},\textcolor{red}{+3},\textcolor{red}{+3}, \textcolor{red}{+3}$
\\
Half Diminished Flat 9: $\textcolor{red}{+3},\textcolor{red}{+3},\textcolor{blue}{+4},\textcolor{red}{+3}$
 
\end{example}

\noindent From a pure mathematical perspective, one could propose chords associated to any set of generators and set the major/minor names with respect to the generators considered. However, we should follow some rules to eliminate what can be deemed as trivial or redundant. For example $\mathbb{Z}_{12}=\langle 1 \rangle$ or $\mathbb{Z}_{12}=\langle 9,4 \rangle$ or $\mathbb{Z}_{12}=\langle 3,8 \rangle$. 
\begin{remark} 
In Western classical music, certain permutations of the notes within a major/minor chord give rise to more chords. With more rigour, we might have called the chords in Definition \ref{cho} basic chords. For example, the first inversion triad is obtained by applying the cycle permutation $(123)$ to a root position triad. The second inversion triad obtained by applying $(132)$ to the root position triad. Another interesting example is the so-called $C$ heavenly chord. Starting with the $C$ major $9^{\text{th}}$ chord $C-E-G-B-D$, one applies $(12)$ cycle permutation on the first two notes, and the $(132)$ cycle on the last three to obtain $E-C-D-G-B$. We leave the subject of chords obtained by such transformations out, though it provides more clues of the strong connection between group theory and musical expression.  
\end{remark}
\noindent The sound component so far seems removed from the group theory behind chord construction. We will connect with it in section \ref{sound}. Before then, we will explain algebraically the concept ``circle of fifths", display the patterns by which major and minor scales are built, and focus on symmetry properties of the counterpoint. These constructions represent guiding  principles in music composition.

\subsection{The Circle of Fifths}
\begin{definition}
The following ordered sequence $\mathcal{C}:=[0,7,2,9,4,11,6,1,8,3,10,5]$ is called circle of fifths with respect to the chromatic scale $\mathbb{Z}_{12}$.    
\end{definition}
\begin{remark}
It is straightforward to observe the pattern in building the circle above: starting at $C=0$ add $7=4+3$, the sum of the generators $\{3,4\}$ of $\mathbb{Z}_{12}$, successively. Note that as a set, $\mathcal{C}$ spans the group $\mathbb{Z}_{12}$. The reason is that gcd$\{7,12\}=1$ hence $\langle 7 \rangle =\mathbb{Z}_{12}$. In Western classical music, one draws a circle on which the elements of $\mathcal{C}$ are placed as vertices which we call ``keys". 
In Figure 3, the circle of fifths is drawn. From the previous section we easily infer that the name ``fifths" is due to counting on the $C$ major scale. The circle if fifths is connected to musical scales (more below). If we regard the consecutive vertices $F-C-G$ on the circle $\mathcal{C}$ as chords in the key of $C$, then the key is supplanted with more chords, according to the $C$ scale, and with a minor/diminished flavor.  

\end{remark}

\tikzset{
    to/.style={
        ->,
        thick,
        shorten <= 1pt,
        shorten >= 1pt,},
    from/.style={
        <-,
        thick,
        shorten <= 1pt,
        shorten >= 1pt,}
}

 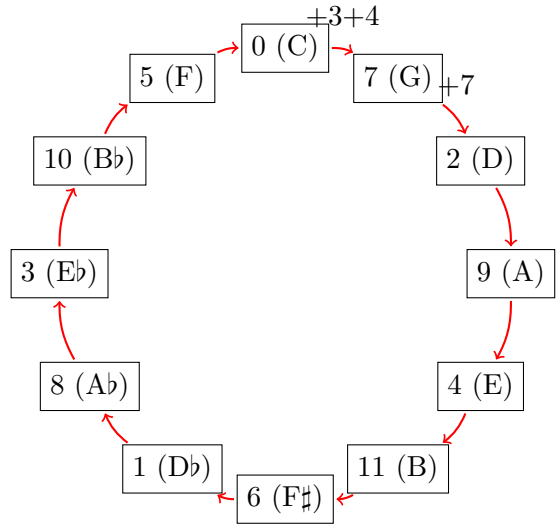
\begin{figure}
   \begin{center}
        \begin{tikzpicture}[scale=3]
           
            \node[shape=rectangle, draw=black] (0) at (0,1) {$0$ (C)};
            \node[shape=rectangle, draw=black] (7) at (1/2,.866) {$7$ (G)};
            \node[shape=rectangle, draw=black] (2) at (0.866,1/2) {$2$ (D)};
            \node[shape=rectangle, draw=black] (9) at (1,0) {$9$ (A)};
            \node[shape=rectangle, draw=black] (4) at (0.866,-1/2) {$4$ (E)};
            \node[shape=rectangle, draw=black] (11) at (1/2,-.866) {$11$ (B)}; 
            \node[shape=rectangle, draw=black] (6) at (0,-1) {$6$ (F$\sharp$)}; 
            \node[shape=rectangle, draw=black] (1) at (-1/2,-0.866) {$1$ (D$\flat$)};
            \node[shape=rectangle, draw=black] (8) at (-0.866,-1/2) {$8$ (A$\flat$)}; 
            \node[shape=rectangle, draw=black] (3) at (-1,0) {$3$ (E$\flat$)};
            \node[shape=rectangle, draw=black] (10) at (-0.866,1/2) {$10$ (B$\flat$)}; 
            \node[shape=rectangle, draw=black] (5) at (-1/2,0.866) {$5$ (F)}; 

            \path (0) edge[to, bend left=15,draw=red] node[label=above:{+3+4}] {} (7);
            \path (7) edge[to, bend left=15,draw=red] node[label=above:{+7}] {} (2);
            \path (2) edge[to, bend left=15,draw=red] node {} (9);
            \path (9) edge[to, bend left=15,draw=red] node {} (4);
            \path (4) edge[to, bend left=15,draw=red] node {} (11);
            \path (11) edge[to, bend left=15,draw=red] node {} (6);
            \path (6) edge[to, bend left=15,draw=red] node {} (1);
            \path (1) edge[to, bend left=15,draw=red] node {} (8);
            \path (8) edge[to, bend left=15,draw=red] node {} (3);
            \path (3) edge[to, bend left=15,draw=red] node {} (10);
            \path (10) edge[to, bend left=15,draw=red] node {} (5);
            \path (5) edge[to, bend left=15,draw=red] node {} (0);                                           
        \end{tikzpicture}
        \caption{The Circle of Fifths in Western Classical Music}
     \end{center}
\end{figure}

\begin{figure}
   \begin{center}
        \begin{tikzpicture}[scale=3]
           
            \node[shape=rectangle, draw=black] (0) at (0,1) {$0$};
            \node[shape=rectangle, draw=black] (5) at (0.866,1/2) {$5$};
            \node[shape=rectangle, draw=black] (4) at (0.866,-1/2) {$4$};
            \node[shape=rectangle, draw=black] (3) at (0,-1) {$3$}; 
            \node[shape=rectangle, draw=black] (2) at (-0.866,-1/2) {$2$}; 
            \node[shape=rectangle, draw=black] (1) at (-0.866,1/2) {$1$}; 

            \path (0) edge[to, bend left=25,draw=red] node[label=above:{+2+3}] {} (5);
            \path (5) edge[to, bend left=25,draw=red] node[label=right:{+5}] {} (4);
            \path (4) edge[to, bend left=25,draw=red] node {} (3);
            \path (3) edge[to, bend left=25,draw=red] node {} (2);
            \path (2) edge[to, bend left=25,draw=red] node {} (1);
            \path (1) edge[to, bend left=25,draw=red] node {} (0);                                   
        \end{tikzpicture}
        \caption{The Circle of Fifths for $\mathbb{Z}_{6}=\langle 2,3\rangle$}
   \end{center}
\end{figure}
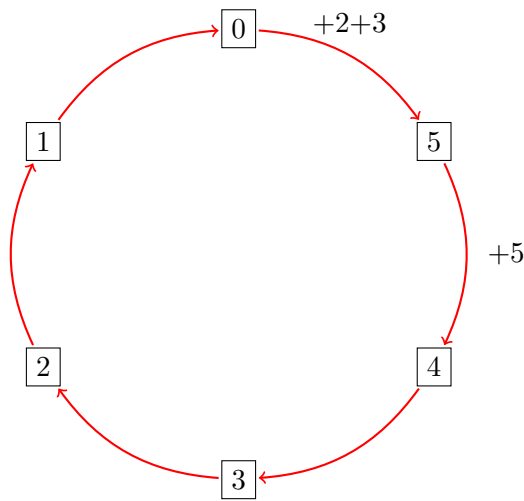

\begin{figure} 
    \begin{center}  
        \begin{tikzpicture}[scale=3]
           
            \node[shape=rectangle, draw=black] (0) at (0,1) {$0$};
            \node[shape=rectangle, draw=black] (7) at (0.6,0.8) {$7$};
            \node[shape=rectangle, draw=black] (4) at (0.95,.3) {$4$};
            \node[shape=rectangle, draw=black] (1) at (0.95,-.3) {$1$}; 
            \node[shape=rectangle, draw=black] (8) at (0.6,-.8) {$8$}; 
            \node[shape=rectangle, draw=black] (5) at (0,-1) {$5$}; 
            \node[shape=rectangle, draw=black] (2) at (-0.6,-0.8) {$2$}; 
            \node[shape=rectangle, draw=black] (9) at (-0.95,-.3) {$9$}; 
            \node[shape=rectangle, draw=black] (6) at (-0.95,0.3) {$6$}; 
            \node[shape=rectangle, draw=black] (3) at (-0.6,0.8) {$3$}; 

            \path (0) edge[to, bend left=20,draw=red] node[label=above:{+2+5}] {} (7);
            \path (7) edge[to, bend left=20,draw=red] node[label=right:{+7}] {} (4);
            \path (4) edge[to, bend left=20,draw=red] node {} (1);
            \path (1) edge[to, bend left=20,draw=red] node {} (8);
            \path (8) edge[to, bend left=20,draw=red] node {} (5);
            \path (5) edge[to, bend left=20,draw=red] node {} (2);
            \path (2) edge[to, bend left=20,draw=red] node {} (9);
            \path (9) edge[to, bend left=20,draw=red] node {} (6);
            \path (6) edge[to, bend left=20,draw=red] node {} (3);
            \path (3) edge[to, bend left=20,draw=red] node {} (0);

        \end{tikzpicture}
        \caption{The Circle of Fifths for $\mathbb{Z}_{10}=\langle 2,5\rangle$}
   \end{center}
\end{figure}
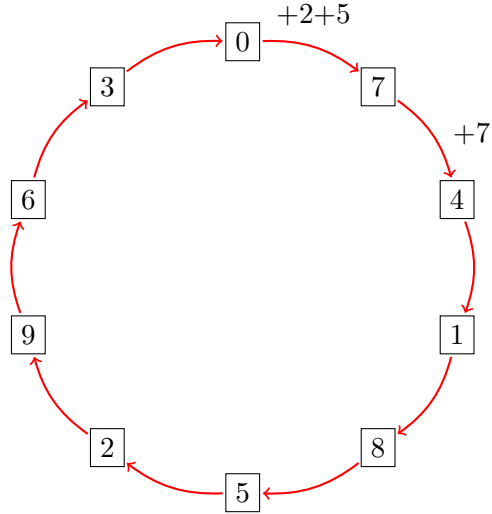

\begin{remark} We will generalize the concept ``circle of fifths" later on, and prove that it makes sense for certain $\mathbb{Z}_n$ and generators. The concept is equivalent to the one found in \cite{CM}, although we arrive at it via group generators. For example, in Figure 4 and 5 we draw the analogous concept for the groups $\mathbb{Z}_6=\langle 2,3\rangle$, $\mathbb{Z}_{10}=\langle 2,5\rangle$, and the corresponding summation of their generators.  The circle can be ``trivial", by which we mean musically uninteresting: for example the generator $1$ (there is nothing to add but $+1$ ) gives rise to a circle of consecutive notes, similar to the unoriented Cayley graph corresponding to $\langle 1\rangle $. Playing consecutive notes or chords may sound predictable. Another example is $\mathbb{Z}_{6}=\langle 3,4 \rangle$, because $(3+4 ) \text{mod } 6= 1$, so a circle of fifths based on these generators is trivial. It is also possible that other pairs  of generators define non-trivial circles; e.g. both cases $\mathbb{Z}_{10}=\langle 2, 5\rangle$ and $\mathbb{Z}_{10}=\langle 8, 5\rangle$. 
We caution the reader that the circles may not be comparable as sets for different groups $Z_n$; their construction depends on the generator set and the mod-$n$ binary operation. The group structure may not be compatible, e.g. $\mathbb{Z}_{10}$ is not a subgroup of $\mathbb{Z}_{12}$. 
\end{remark}

\subsection{Major and Minor Scales}
The development, evolution and explanation of musical scales is complicated. See for example \cite{Gan} for more insights with respect to various factors (historical, cultural) that influenced and helped refine this subject. Our view aims at unifying it with chords, circle of fifths, and later counterpoint, under the same roof represented by the algebraic and geometric features of the group $\mathbb{Z}_{12}$. The patterns we identify help define scales for other groups $\mathbb{Z}_n$ in section \ref{systems}. These are based again on the generators of the group and major/minor scales. We mention that in \cite{CM} generalized scales are constructed with respect to each pitch class $k$ relatively prime to $n$. This construction, however, does not display major/minor flavors. 
\begin{definition}
 Let $x\in\mathbb{Z}_{12}$. \\
 i) The $x$ major scale is the sequence
 $x, x\oplus 2,  x\oplus 4,  x\oplus 5,  x\oplus 7,  x\oplus 9,  x\oplus 11, x\oplus 12=x$. \\
 ii) The $x$ minor scale is the sequence
 $x, x\oplus 2,  x\oplus 3,  x\oplus 5,  x\oplus 7,  x\oplus 8,  x\oplus 10, x\oplus 12=x$.
\end{definition}
\noindent We defined the major/minor scale as a loop at note $x$. However, when playing the scale on an instrument, the ending is placed an octave higher.  
 \begin{example}
 The $C$ major scale is the familiar sequence of notes $C, D, E, F, G, A, B, C$. The $C$ minor scale is the sequence $C, D, E\fl, F, G, A\fl, B\fl, C$.   
 \end{example}

\begin{remark}\label{scale}
The definition might seem peculiar and needs some justification. Given $x$ we will apply the following steps to ``populate" an $x$ major and minor scales sequence. In these steps we simply observe that with $k=2$, $4=2\cdot k$ is even (and say it displays the $'22'$ pattern), and with $l=1$, $3=2\cdot l+1=1+2\cdot l$ is odd (and say it displays the $'21'$/ $'12'$ patterns). These simple ideas will enable us to generalize scales  construction later in section \ref{systems}. \\
$\bullet$ The sequence must contain the $x$ major (minor) seventh chord. Hence the starting subsequence is $x$, $x\oplus 4$, $x\oplus 7$, $x\oplus 11$  ($x$, $x\oplus 3$, $x\oplus 7$, $x\oplus 10$ for minor). \\
$\bullet$ If the counting distance between two subsequence elements is $\leq 2$ then there is nothing to add in between. Also, close the loop with $x\oplus 12$ in both major/minor subsequences. All other counting distances (not including the endpoint) between two consecutive elements in the subsequence is either $3$ or $4$. \\
$\bullet$ 
 For the major scale: to the subsequence thus far, $x$, $x\oplus 4$, $x\oplus 7$, $x\oplus 11$, $x\oplus 12$, add elements in such a way that the following patterns emerge between consecutive, updated elements: $'22'$ if the counting distance between two points is $4$, and $'12'$ in case it is $3$. Hence, for the leg $x, x\oplus 4$ only $x\oplus 2$ is needed because $x, x\oplus 2, x\oplus 4$ shows the $'22'$ distance pattern. Between  $x\oplus 4$, $x\oplus 7$, aiming at the $'12'$ pattern we add $x\oplus (4+1)$. Between $x\oplus 7$ and  $x\oplus 11$ we need add $x\oplus 9$ to obtain the $'22'$ pattern. The counting distance pattern is the familiar $'22\text{ }12\text{ }22 \text{ }1'$.  \\
$\bullet$ For the minor scale: add elements to the subsequence thus far according to the patterns $'22'$ in case the counting distance is $4$, and an alternate $'21'/'12'$ patterns in case it is $3$. Similar analysis reveals the $x$ minor scale in the definition above. Notice however the somehow whimsical alternation of the $'21'$ pattern ( for the leg $x, x\oplus 2, x\oplus 3$) and the $'12'$ pattern (for the leg  $x\oplus 7, x\oplus 8, x\oplus 10$). The minor scale counting distance pattern is the familiar $'21\text{ }22\text{ }12 \text{ }2'$.
 \end{remark}

\subsection{Counterpoint}
Counterpoint represents a sum of composing techniques which combine two or more voices. These techniques have arisen and evolved within the Western classical music body since the $9^{\text{th}}$ century. A first systematic compilation is presented in \cite{Fux}, where  the rules of composing with counterpoint are spelled out. For a self contained treatment we refer for  example to \cite{Rus}. We will focus on the first species of counterpoint, and explain in detail its connection with the concepts and results presented in the last part of Section \ref{groups}. 
Our group theory point of view is inspired by \cite{Mco}, where the group $\mathbb{Z}_3\times \mathbb{Z}_4$, which is isomorphic to $\mathbb{Z}_{12}$, is used instead. Its unoriented Cayley graph, obtained with generators $(1,0)$ and $(0,1)$, is called a discrete torus. It is the same graph as the one in Figure 2, without arrows (unoriented Cayley graph). The three dimensional ``torus"  can be obtained by ``pulling out" some of the planar edges so that the (imaginary) edge crossing effect disappears. 
The fact that the the group of affine transformations  $\text{Aff}(\mathbb{Z}_{12})$  acts isometrically on the torus of thirds was already noted in \cite{Mco} and proved rigorously in Agustín-Aquino’s master thesis.

\begin{definition}\label{cptd} The elements of the set $K\subset \mathbb{Z}_{12}$, $K:=\{0,3,4,7,8,9\}$, are called consonants. We say that its complement in $\mathbb{Z}_{12}$, $\{1,2,5,6, 10,11 \}$, consists of dissonant elements and denote this set by $D$. The partition $(K,D)$ is called Fux dichotomy.
\end{definition}
\noindent The main features of the first species of counterpoint are described below: \\
$\bullet$ If voice A plays $ x, y, z...$ then voice  
B plays$: x\oplus k_1, y\oplus k_2, z\oplus k_3... $ with $k_i\in K$, $i=1,2,3...$\\
$\bullet$ There are restrictions: ``parallel" fifths are forbidden,  i.e. consecutive distances $k_i, k_{i+1}$ in the sequence above cannot be both equal to $7$. For example, the distances $3,4$ (minor/major third) and $8,9$ (minor/major sixth) ``are fine but no more than three in a row", see e.g. \cite{Rus}.\\
Let us mention that the treatment below and its generalization in Section \ref{systems} covers only the consonant/dissonant paradigm, and not the exceptions present in the various species of counterpoint. \\
\\
There has been for a long time, a discussion among music theorists
about considering the perfect fourth as consonant, i.e. add $5$ to $K$. In \cite{Mco} the choice of the partition $(K, D)$ is explained through actions of symmetries $T$ on the unoriented Cayley graph. More precisely, the three properties  
\begin{equation}\label{cpt}
T^2=Id,\quad T(K)=D,\quad T\text{ is an isometry}
\end{equation}
 of the affine transformation $T(x)=5x\oplus 2$ are interpreted as a clue that $5$, the fourth, should stay dissonant. Using the tools from Section \ref{groups}, we are in position to prove the theorem below, which is the result mentioned in \cite{Mco}. See also \cite{AJM} for an in depth analysis and examples of multiple counterpoint partitions that are possible due to the action of the linear group of affine transformations on $\text{Aff}(\mathbb{Z}_{2n})$.  
 \begin{theorem}\label{cpt12} There exists exactly one affine transformation $T$ on the unoriented Cayley graph of $\mathbb{Z}_{12}$ generated by $S=\{3,4,8,9\}$, satisfying the conditions  (\ref{cpt}). 
More precisely, this transformation is given by $T(x)=5x\oplus 2$ for $x\in\mathbb{Z}_{12}$.
 \end{theorem}
 \begin{proof} It can be checked directly that $T(x)=5x\oplus 2$ does satisfy the required properties, however, we need to prove that this is the only transformation satisfying (\ref{cpt}). Let $T:\mathbb{Z}_{12}\to \mathbb{Z}_{12}$, $T=f\oplus w$ be an affine transformation, hence $f\in\text{Aut}(\mathbb{Z}_{12})$ and $w\in \mathbb{Z}_{12}$ (because $\mathbb{Z}_{12}$ is abelian, right and left translations coincide). From Example \ref{aut} we see that $f$ is necessarily one of the automorphisms $\{f_1, f_2, f_3, f_4\}$. 
  We will treat each case separately, but we notice first that by Corollary \ref{cor} and Example \ref{isom}, $T$ is an isometry in each of the four cases. It remains to select the ones that satisfy $T^2=Id$ and $T(D)=K$.
 \\
 {\bf{Case 1}}. $f(x)=x$ $\quad\forall$ $x\in \mathbb{Z}_{12}$. Then $T(x)=x\oplus w$. Because $T^2=Id$, in particular  $T^2(0)=0=w\oplus w$. Hence $w\in \{0,6\}$. However, none of the transformations $T(x)=x$ and $T(x)=x\oplus 6$ satisfies $T(K)=D$. We discard this case. 
 \\
 {\bf{Case 2}}. $f(x)=5\odot x$ $\quad\forall$ $x\in \mathbb{Z}_{12}$. Then $T(x)=5x\oplus w$. Because $T^2(0)=0$ we must have   $6w\text{ mod } 12=0$. Hence $w\in \{0,2,4,6,8,10\}$. The requirement  $T(K)=D$ rules out all values except $w=2$. Also, by a direct check (or the comment in Example \ref{aut}, or by Proposition \ref{ccpt} ) we have that with $w=2$, $T^2=Id$. Hence we found an affine transformation $T$  that satisfies all three requirements in (\ref{cpt}). 
 \\
 {\bf{Case 3}}. $f(x)=7\odot x$ $\quad\forall$ $x\in \mathbb{Z}_{12}$.
 Then $T(x)=7x\oplus w$. Because $T^2(0)=0$, we must have   $8w\text{ mod } 12=0$. Hence $w\in \{0,3,6,9\}$. None of these values $w$ corresponds to $T(K)=D$. E.g. when $w=3$, $T(0)=3\notin D$; when $w=6$, $T(4)=7\notin D$; when $w=6$, $T(3)=3\notin D$; when $w=9$, $T(0)=9\notin D$. Hence, we discard this case.\\ 
{\bf{Case 4}}. $f(x)=11\odot x$ $\quad\forall$ $x\in \mathbb{Z}_{12}$.  Then $T(x)=11x\oplus w$. Notice in this case $T^2(x)=121x\oplus 12w=121x\text{ mod }12=x$, hence any $w$ in $\mathbb{Z}_{12}$ could do. Because $T(0)=w$, the values $w\in\{0,3,4,7,8,9\}$ are ruled out. The remaining ones $w\in\{1,2,5,6,10,11\}$ are ruled out as follows: if $w=1$ then $T(4)=9\notin D$; if $w=2$ then $T(7)=7\notin D$; if $w=5$ then $T(8)=9\notin D$; if $w=6$ then $T(3)=3\notin D$; if $w=10$ then $T(3)=7\notin D$; if $w=11$ then $T(3)=8\notin D$. \\
From the four cases above we conclude only $T(x)=5\odot x\oplus 2$ satisfies (\ref{cpt}). 
\end{proof}

\begin{remark} The theorem above does not ``create" the Fux partition, though it asserts an important symmetry property. One would like to have a rigorous principle by which such partition arises, especially if one wants to extend counterpoint to other $\mathbb{Z}_n$. In \cite{AJM}, by analysing what affine transformations correspond to suitable partitions such that (\ref{cpt}) holds, a plethora of counterpoint partitions (called  strong dichotomies) is found. From a mathematical perspective such abundance may not point out the Fux partition as special. We could filter out a few if we impose restrictions. The reader has surely noticed that the consonant set $K$ contains the generators $3,4$ and their inverses $9$ and $8$ in $\mathbb{Z}_{12}$. These elements represent paths of length $1$ in the unoriented Cayley graph. Thus, the set  $\{0,3,4,8,9\}$ of consonants is made of minimum distances between two voices, on the unoriented Cayley graph. One may wish to formulate counterpoint as a simple minimization principle, however because $7\in K$, representing paths of length $2$ ($7=3+4$), one would have to either add all such paths (e.g. $5=(9+8)\text{ mod }12$, $6=3+3$ encode paths of length $2$) or remove $7$ from $K$. Either way, we would be led to drop (\ref{cpt}), because in particular, $T$ being bijective, we get $|D|=|T(K)|=|K|$, i.e. $D$ and $K$ must have the same number of elements. 
\end{remark}
\noindent We want to arrive at the important symmetry conditions in (\ref{cpt}) while keeping the min distance elements (i.e. the generating symmetric set $S=\{3,4,8,9\}$) in the consonant set. In this way, we will ``create" all possible partitions $(K,D)$ of $\mathbb{Z}_{12}$ that satisfy (\ref{cpt}) for some affine $T$. Let $K':=S\cup\{0\}=\{0,3,4,8,9\}$. We want $K'\subset K$. Then $T(K')\subset T(K)=D$ and $K'\cap D=\emptyset$ because $K\cap D=\emptyset$; hence $T(K')\cap K'=\emptyset$.  
To find $K$ (and by default $D=\mathbb{Z}_{12}\setminus K$) we need decide only one more element besides those in $K'$, i.e. $K=K'\cup\{z\}$ where $z\in\{1,2,5,6,7,10,11\}$ is to be found. This will be done by asking what affine $T$, if any, satisfies:
\begin{equation}\label{wcpt}
T^2=Id,\quad T(K')\cap K'=\emptyset,\quad T\text{ is an isometry on the unoriented Cayley graph}
\end{equation}
Using the first and third conditions in (\ref{wcpt}) and the ideas in the proof of Theorem \ref{cpt} we have to have $T=f\oplus w$ for some $f\in \text{Aut }(\mathbb{Z}_{12})$ and  $w\in\mathbb{Z}_{12}$. Because $T(0)=w$ it follows from (\ref{wcpt}) that $w\notin K'$, hence $w\in\{1,2,5,6,7,10,11 \}$.   
At the same time, $f$ must be one of the automorphisms from Example \ref{aut}. We discuss the same four cases from the proof of Theorem \ref{cpt12}: If $f(x)=x$ then $w\in\{0,6\}$. From above restrictions on $w$ we are left with $w=6$. This value, however, violates $T(K')\cap K'=\emptyset$ because we would get $T(3)=9\in T(K')\cap K'$.  
If $f(x)=5\odot x $ then $w\in\{2, 6, 10\}$. We rule out $w=6$ because we would get $T(3)=15\oplus 6=9\in T(K')\cap K'$.   
The isometry $T(x)= 5x\oplus 2$ yields the partition sets   
$K=\{0,3,4,8,9,z\}$, $D=T(K)=\{ 2,5,10,6,11, T(z) \}$ where $z, T(z)\in\{1,7\}$. There are two choices for $z\neq T(z)$, hence we obtain two partitions that satisfy  (\ref{cpt}), for the uniquely found $T$:
\begin{align}
 K_1 &=\{0,3,4,8,9,7\}, \quad  D_1=\{ 2,5,10,6,11, 1 \}\label{p1} \\
 K_2 &=\{0,3,4,8,9,1\}, \quad D_2=\{ 2,5,10,6,11, 7\}\label{p2}
\end{align}
Similarly, the isometry $T(x)= 5x\oplus 10$ brings about the following partitions 
\begin{align}
K_3 &=\{0,3,4,8,9,5\}, \quad D_3=\{10,1,6,2,7,11 \}\label{p3} \\
K_4 &=\{0,3,4,8,9, 11\}, \quad D_4=\{10,1,6,2,7,5 \}\label{p4}
\end{align}
In case 3, $f(x)=7\odot x$ but now $w$ is restricted to $w=6$. Then $T(x)=7x\oplus 6$ has fixed point $x=3$ hence $T(K')\cap K'\neq\emptyset$, and no new partition arises. In case 4, $f(x)=11\odot x$. A thorough check against all values $w\in \{1,2,5,6,7,10,11\}$ shows no new partition $(K,D)$ is produced because either $T(K')\cap K'\neq\emptyset$ or $T$ has a fixed point. E.g., for the value $w=10$, the affine map $T(x)=11x\oplus 10$ does satisfy (\ref{wcpt}), but it has two fixed points $T(5)=5$ and $T(11)=11$. Hence we can't find $z$ such that $K=K'\cup\{z\}$ and $K\cap T(K)=\emptyset$. We summarize this discussion in the form of the following 
\begin{theorem}\label{4p} Let $\mathbb{Z}_{12}$ be generated by the symmetric set $S=\{3,4,8,9\}$, and let $K':=\{0\}\cup S$. There exist exactly four partitions $(K_i, D_i)_{i=1}^4$ of $\mathbb{Z}_{12}$ such that $K'\subset K_i$ and $\exists!$ $T_i:\mathbb{Z}_{12}\to \mathbb{Z}_{12}$  affine transformation satisfying (\ref{cpt}) with respect to the partition $(K_i, D_i)$,  $\forall$ $i\in\{1,2,3,4\}$. The partitions, given in (\ref{p1}), (\ref{p2}) correspond to the unique affine map $T(x)=5x \oplus 2$ , and in (\ref{p3}), (\ref{p4}) to the unique affine map $T(x)=5x \oplus 10$.
\end{theorem}
\begin{definition}\label{condition} i) We say a partition $(K,D)$ of $\mathbb{Z}_n$ satisfies the counterpoint condition whenever (\ref{cpt}) holds with respect to a unique affine $T:\mathbb{Z}_n\to\mathbb{Z}_n$.\\
ii) We say that an affine $T:\mathbb{Z}_n\to\mathbb{Z}_n$ satisfies the weak counterpoint condition if (\ref{wcpt}) holds with respect to the set $K'=\{0\}\cup S$ where $S$ is a symmetric, generating set of $\mathbb{Z}_n$. 
\end{definition}
\noindent We will study the (weak) counterpoint conditions in certain groups $\mathbb{Z}_n$ in Section \ref{systems}. The idea is to find $T$ which satisfy  (\ref{wcpt}), and similar to the proof above increase $K'$ to reach (\ref{cpt}).
\begin{remark}
Notice that the values $7$, $1$, $5$ and $11$ minimize the distance function over all path lengths $\leq 2$ in the unoriented Cayley graph. From the point of view of (\ref{cpt}), all four partitions in Theorem \ref{4p} should be valid to experiment counterpoint with. Although $K_3$ from (\ref{p3}) does justice to $5$ as a would be consonant, one can rule out partitions $(K_{2,3,4}, D_{2,3,4})$ if we minimize further. After imposing the set $K'$ be consonant, we will add $z\notin K'$ to $K$ only if (\ref{cpt}) holds and the path length of $z$ is minimum  with respect to the {\bf{oriented}} Cayley graph (in this graph, the path from $0$ to $9$ has length $3$, and this is why we have to minimize outside $K'$). Then, we would be led naturally to the original counterpoint partition from Definition \ref{cptd}. 
\end{remark}

\section{Sound and Tuning}\label{sound}
\noindent In this section, we present a bare minimum needed to understand how the various concepts from the previous sections can be implemented practically and be recognized when music is produced. Aiming at an efficient simplicity, we avoid going deeper into topics that explain sound through Fourier analysis, differential equations or physical characteristics of musical instruments. The mathematical environment we need here is one-dimensional. In software implementations (we use Maple), a one dimensional array encoding a periodic function is sampled within a time interval and transformed into an audio file. 
\begin{definition}\label{sw} Let $f_0>0$ be a real number, $k$ a non-negative integer, and  $r:=\sqrt[12]{2}$. The function $f:[0,1]\to [-1,1]$, $f(t)=\sin (2\pi f_0 t)$ is a sound wave of frequency $f_0$.   
 The pitch $k$ is the frequency $r^k f_0$ of the sound wave $f_k(t):=\sin (2\pi r^k f_0 t)$.
\end{definition}
\begin{remark} The above sequence of pitches in geometric progression, is called ``equal temperament" in Western music. The frequency ratio between consecutive pitches is constant $r=\sqrt[12]{2}=2^{1/12}$. Obviously, the constant $r$ was chosen so that the chromatic scale displays double frequency length $2f_0$. The starting frequency $f_0$ is usually chosen so that the piano middle $A$ note plays at $440$ Hz frequency. Temperate tuning has the advantage that a song played within a key sounds similar when shifted to another key. Mod $12$ equivalent pitches represent the same note, and their frequency ratio is a power of $2$. The functions $f_k$ model the so-called pure tones. We have made the choice to use the sinus function to model the pure tones, but any periodic function would suffice to implement the musical theory concepts we are concerned with. 
\end{remark}
\noindent Throughout its history, Western music has invented many flavors of tuning. For example, in Pythagorean tuning the ratio between two consecutive pitches is not constant, though this type of tuning is a mathematical approximation, sometimes coarse, of the temperate one (as is the so-called just tuning). Pythagorean tuning presets the values of the fourth and the fifth intervals at $\frac{4}{3}$ and $\frac{3}{2}$, respectively; also, the minor second has frequency $256/243$ which is $\approx 2^{1/12}$; the major second is measured at $9/8\approx 2^{2/12}$; and so on, the interval $i$ is represented by a ratio in the form $2^p3^q$ which roughly approximates   $2^{i/12}$. However, these approximations imply larger errors will accumulate when shifting the notes to other octaves. Computers obviously, approximate irrational number frequencies such as $2^{1/12}$, with rational numbers; nevertheless, the algorithms used yield far better approximations than those based on representation in ratios $2^p3^q$ (Pythagorean) or $2^p3^q5^t$ (just), with integers $p,q,t$.    
For details on how the Pythagorean and just tuning ratios are obtained by some elementary algebra manipulations, see \cite{Pi}. 
\begin{center}
  \begin{tabular}{ | l | l | l | l | }
    \hline
    n=Note  & Interval $[0,n]$   & Pythagorean frequency & Temperate frequency\\ \hline
    0=C    & unison or perfect eighth    & 1   &1                       \vspace{0.02in}  \\
    $1=C\sh$ &  minor second & $256/243$ & $2^{1/12} $ \vspace{0.02in}  \\ 
$2=D$ & major second &  $9/8$ & $ 2^{2/12}$ \vspace{0.02in}  \\
$3=D\sh$ &  minor third & $32/27$ & $2^{3/12}$ \vspace{0.02in}  \\
$4=E$ &  major third  & $81/64$ & $2^{4/12}$  \vspace{0.02in}  \\
$5=F$ & fourth & $4/3$ & $2^{5/12}$ \vspace{0.02in}  \\
$6=F \sh$ & tritone & $729/512$ & $2^{6/12}$  \vspace{0.02in}  \\
$7=G$ & fifth  & $3/2$ & $2^{7/12}$  \vspace{0.02in}  \\
$8= G\sh$ & minor sixth & $128/81$ & $2^{8/12}$  \vspace{0.02in}  \\
$9=A$  & major sixth & $27/16$ & $2^{9/12}$  \vspace{0.02in}  \\
$10= A\sh$  & minor seventh & $16/9$ & $2^{10/12}$  \vspace{0.02in}  \\
$11= B$  & major seventh & $243/128$ & $2^{11/12}$ \\    \hline
  \end{tabular}
\end{center}
To create a pure tone to be played $t$ seconds at frequency $f$, we use the Maple code below, which samples the function $\sin{(2\pi f x )}$ at $44100$ values per second in $[0, t]$.
\begin{lstlisting}
with(AudioTools):
Tone := proc(f,t)
 local x, final;  
 final:= Create( (x) -> evalhf(sin(x/44100*2*Pi*f)), duration=t);
 return final:
end proc:
\end{lstlisting}
Temperate tuning is implemented recursively, with each note $k$  encoded as Tone$(f_0 r^k, t)$, where $r=2^{1/12}$. We discard here mod $12$ equivalence of notes because we want to have access to as many octaves as possible. In practice, computers and instruments are subject to physical limitations. In our Maple code, we improve notes by multiplying pure tones with a so-called ``attack-decay-sustain-release" envelope function $g(x)$, and by adding modulation (shift with variable phase). Thus, the function that encodes note $k$ is of the form $\text{note}_k(x)=g(x)\cdot f_k(x+f_k(x))$, with $f_k$ from Definition \ref{sw}. Finally, major/minor chords are created by a weighted average of the notes within the chord.  

\section{Musical Systems in $\mathbb{Z}_n$}\label{systems}
\noindent  In this section, we extend temperate tuning while keeping it coupled with the Cayley graph structure of the group $Z_n$, when $n$ displays a suitable factorization. We divide the space of a generalized octave in equal frequency intervals. This procedure is not new, although most examples are based on splitting up the pitch space in a multitude of ratios (see \cite{Gan} and references therein for non-twelve divisible equal temperaments, and \cite{Bo}, \cite{Pie} for different length octave, a ``tritave"). What we add, in essence, is the geometric group structure of the group $\mathbb{Z}_n$ with a suitable Cayley graph. This mix allows us to construct compatible chords, scales, circle of fifths and first species  counterpoint partitions.  

\begin{definition} A musical system is a pair $(\mathbb{Z}_n, s)$, where $n\in\mathbb{N}$, $n\geq $ and $s$ is a positive real number, $s>1$, with equidistant frequency intervals $[s^{i/n}, s^{(i+1)/n}]$,  $i=0,..., n-1$.  The pure tones attached to the musical system are defined by $f_i(t)=\sin (2\pi f_0 s^{i/n} t)$, for a fixed frequency $f_0>0$. 
\end{definition}
\begin{remark}
 Obviously for $n=12$ and $s=2$ one obtains the temperate $12$ tone chromatic scale. For $n=13$ and $s=3$ we obtain the Bohlen-Pierce ``tritave", however, we have no further analogy because $13$ is prime; we attach the $\mathbb{Z}_n$ group structure and its Cayley graph features to a musical system when $n$ is a product of two relatively prime numbers. We will restrict to this particular class of integers while keeping the octave length $s>1$  arbitrary.   
\end{remark}
\begin{theorem}\label{mt1}
Let $p,q \in \mathbb{N}$ with gcd $\{p,q\}=1$, and $n=pq$.
Then $\langle p,q\rangle = \langle p, n-q \rangle = \langle q,n-p\rangle = \langle n-p,n-q\rangle = \mathbb{Z}_n$
\end{theorem}
\begin{proof}
This is a simple consequence of Theorem 8.2 and its Corollary 2 in \cite{Ga}:  $\mathbb{Z}_a \times \mathbb{Z}_b \cong \mathbb{Z}_n$ if and only if $n = ab$ and gcd$\{a,b \}=1$
\end{proof}
\noindent The theorem above allows us to define chords. According to the theorem, there are four generator pairs, and at this point each one could be used to define chord sequences. We will need the relatively prime setup for the circle of fifths existence, hence the generator set $\{p,q\}$ is most suitable.   
\begin{definition} Let $p,q \in \mathbb{N}$ with gcd $\{p,q\}=1$, $n=pq$, and assume $p>q$. For $x\in\mathbb{Z}_{n}$ the $x-$major triad chord is the path $x\to x+p\to x+p+q$. The $x-$minor triad chord is the path $x\to x+q\to x+q+p$.  
More generally, an $x-$major (minor) chord is a path $x_1\to x_2\to\cdots \to x_k$ in the oriented Cayley graph such that $x=x_1$ and:  \\
i) $x_{2}=x_{1}\oplus p$ ( $x_{2}=x_{1}\oplus q$ for minor )\\
ii) $x_{3}=x_{2}\oplus q$ ( $x_{3}=x_{2}\oplus p$ for minor )\\
iii) $x_{i+1}\ominus x_i\in\{p,q\}$, $\forall$ $i=1,...,k$\\
iv) the path is non-self intersecting unless $x_k=x_1$.\\
An $x$ chord $x_1=x\to x_2\to ...\to x_k$ is said to be within the octave if 
$\displaystyle{\sum_{i=1}^{k-1} x_{i+1}\ominus x_i\leq n}$. 
\end{definition}
\begin{example}
In $Z_{10}=\langle 2,5\rangle $ the $0-$major triad is the path $0\to 5\to 7$. In $\mathbb{Z}_{15}=\langle 3,5\rangle $ the path $0\to 5\to 8\to 13$ represents a $0-$major chord. In $\mathbb{Z}_{20}=\langle 4,5\rangle $, the $0-$minor triad is the path $0\to 4\to 9$. We warn the reader that for different values $n$, notes labelled by the same symbol are not equivalent and do not sound the same. We start by convention with $0$ representing the same frequency sound in all systems $(\mathbb{Z}_n, s)$. With $s$ kept fixed, a note $k$ in $(\mathbb{Z}_{n}, s)$ encodes a sound wave at frequency $s^{k/n}$, whereas the same $k$ in a system $(\mathbb{Z}_{m}, s)$ encodes a sound wave at frequency $s^{k/m}$. Hence, the $0-$ major triad $0\to 5\to 7$ of $(\mathbb{Z}_{10}, 2)$ and the $0-$major triad $0\to 5\to 8$ of $(\mathbb{Z}_{15}, 2)$ have only the sound of note $0$ in common, whereas the sound of $k=5$ will depend on whether $n=10$ or $n=15$.
\end{example}
\begin{example}
In $\mathbb{Z}_{12}$ the $C$ major triad $0\to 4\to 7$ is within the octave. The largest $C$ major chord with this property is $0\to 4 \to 7\to 11$, i.e. the $C$ major $7^{\text{th}}$ chord. The $C$ major $9^{\text{th}}$ chord, $0\to 4 \to 7\to 11\to 2$ is not within the octave because $4+3+4+3>12$. In $\mathbb{Z}_{15}=\langle 3,5\rangle$, the minor chord $1\to 4\to 9\to 12 $ is the largest $1-$minor chord within the octave. Notice that $1\to 4\to 9\to 12\to 2 $ is also a $1-$minor chord, but not within the octave because $3+5+3+5>15$. We will use the ``largest major/minor chord within an octave" concept later, to define major/minor scales. 
\end{example}

\begin{remark}
  One can consider a more general factorization $n= p_1 p_2\dots p_k$ in pairwise relatively prime factors. However, one may encounter issues with duplicate generating sets, such as $Z_{30}=\langle 2,3,5 \rangle =\langle 5,6\rangle$. Another issue might arise in the circle of fifths construction, for which we would need $n$ and $p_1+p_2+...+p_k$ to be relatively prime. One can still implement chords based on multiple generators and explore more means of musical expression. For example, with generators $\{2,3,5\}$ one obtains a variety of chords that contains mixed flavors of major and minor ones.  Such cases can be coupled with a longer octave length, so that the $30$ notes within the scale are well-spaced out.     
\end{remark}

\noindent Next theorem justifies a construction of the generalized circle of fifths. Provided $n$, $p$ and $q$ are chosen as in the theorem,  starting with $0$ and adding $p+q$-steps, all elements of $\mathbb{Z}_n$ lie  on a circle. In other words $\langle p+q \rangle = \mathbb{Z}_n$. Moreover, the circle is not ``trivial" to the right, in the sense that, clockwise the circle does not display the elements of $\mathbb{Z}_n$ consecutively. In \cite{CM}, generalized circles of fifths are defined for each $k\in\mathbb{Z}_n$ such that gcd$\{k,n\}=1$. Our version obviously corresponds to $k=p+q$. 
\begin{theorem}\label{mt2}
Let $p,q \in \mathbb{N}$ with gcd $\{p,q\}=1$, $p,q >1$, and $n=pq$. We have:  \\
 gcd $\{ p+q, n \} =1$ and $(p+q)\text{ mod }n = p+q \neq 1$.
\end{theorem}
\begin{proof} Let $k \mid n$ and $k \mid p+q$. We must show $k=1$. Since $n=pq$, $k \mid pq$. Assume by contradiction, $k>1$. Because we know gcd$\{p,q\}=1$, either $k \mid p$ or $k \mid q$ or $\exists$ $p'$, $\exists$ $q'$ prime number divisors of $p$ and $q$, respectively, such that $p'q'\mid k$ (to see this, use the unique prime number factorization of $p$ and $q$). In either of these three cases, because $k\mid p+q$, we would find a common divisor of both $p$ and $q$, which contradicts gcd$\{p,q\}=1$.
The second relation is obvious because $1<p+q<pq$.
\end {proof}
\begin{definition}
The ordered sequence $\mathcal{C}_n:=[i\odot (p+q)]_{i=0}^{n}$ is called circle of fifths with respect to the chromatic scale $\mathbb{Z}_{n}$.    
\end{definition}

\begin{remark} 
In Section \ref{westmusic} we have drawn examples of such circles.   
The circle may be ``trivial" in certain cases, i.e. the sequence $\mathcal{C}_n$ is made of all $\mathbb{Z}_n$'s elements in consecutive order. For example, $\mathbb{Z}_{6}=\langle 3,4 \rangle$ and $(3+4 ) \text{mod } 6= 1$, hence the circle of fifths based on these generators is trivial. It is also possible that other pairs of generators define non-trivial circles: e.g. both cases $\mathbb{Z}_{10}=\langle 2, 5\rangle$ and $\mathbb{Z}_{10}=\langle 8, 5\rangle$ define non-trivial circles of fifths and one can choose either one to define chords.
\end{remark}
\noindent Next we propose a definition of scales in $\mathbb{Z}_n$ akin to Definition \ref{scale}. Recall that we are working in the case $n=pq$ and gcd$\{p,q\}=1$. Without loss of generality, say $p>q$. We notice $p$ and $q$ cannot be both even. However, there are three cases to distinguish depending on whether the pair $(p,q)$ is (odd, even), (odd, odd) or (even, odd). The cases (odd, even) and (even, odd) are not symmetric, due to the major/minor flavor of a scale.

\begin{definition} Let $n=pq$ with gcd$\{p,q\}=1$ and $p>q$. \\
a) Assume first that $p=2k$ and $q=2l+1$. We say that $p$ displays the $'22...2'$ pattern (there are $k$ occurrences of the digit $2$) and $q$ displays the $'22..21'$/ $'12..22'$ pattern (there are $l$ occurrences of the digit $2$ and one of the digit $1$). Let $x\in\mathbb{Z}_n$. \\
i) The $x$-major scale is the sequence $(x_i)_{i=1}^t\subset\mathbb{Z}_n$ such that 
\begin{itemize}
    \item  $x_1=x$, $x_t=x$ and the sequence contains the largest $x-$major chord within the octave.
    \item for any consecutive notes $x_i\to x_j$ of the above $x$-major chord, we have  \\ $x_{i+1}=x_i\oplus 2$,..., $x_{j-1}=x_{j-2}\oplus 2$,   
    $x_j=\begin{cases} 
    x_{j-1}\oplus 2& \text{ if } \quad x_j\ominus x_i=p\\
    x_j=x_{j-1}\oplus 1 & \text{ if } \quad  x_j\ominus x_i=q
    \end{cases}$\\
 (i.e. between consecutive notes of the chord, the patterns are either $'22...2'$ or $'22..21'$).
\end{itemize}
ii) The $x$-minor scale is defined similarly, by considering the largest $x$-minor chord within the octave, and by filling the scale with alternating $'22..21'$/ $'12..2'$ patterns. \\
b) If $p=2k+1$ then the scales are defined similarly, by using the $'\underbrace{22...2}_{k}1'$ pattern for $p$. For the minor chord, if $q=2l+1$ then use its pattern $'\underbrace{22...2}_{l}1'$ without alternating it with $'1{\underbrace{2...22}_{l}}'$; if $q=2l$ then use its $'22...2'$ pattern instead.   
\end{definition}
\noindent The definition above mimics and generalizes  the alternation patterns observed in the major/minor scales of $\mathbb{Z}_{12}$. The purpose of extra alternation in the minor scale (when $p$ is even and $q$ odd), or the lack of it (when $p$ is odd and $q$ is odd) is to avoid $'11'$ occurrences within the scale. We illustrate the definition with a few examples of $0-$major/minor scales. By translation, one can obtain all $x$ scales within a musical system.

\begin{example}  For $\mathbb{Z}_{10}=\langle 2, 5\rangle$, $p=5$, $q=2$. The patterns are $p=\text{}'221'$ and $q=\text{}'2'$. $0\to 5\to 7$ is the largest $0-$major chord within the octave. In between $x_1=0$ and $x_4=5$,
 we fill in $x_2=2$, $x_3=4$, hence the leg $x_1\to x_4$ satisfies the $'221'$ pattern. The leg $x_4=5\to x_5=7$ satisfies the $q$ pattern by default. Because the scale closes with $x_t=0$, we obtain the $0-$major scale as the sequence $x_1=0, x_2=2, x_3=4, x_4=5, x_5=7, x_6=0$. The $0-$minor scale is built on the minor chord $0\to 2\to 7$. Adding notes using the definition, we obtain the sequence $0,2,4,6,7,0$. 
 \end{example}
 
 \begin{example} For $\mathbb{Z}_{15}=\langle 3, 5\rangle$, $p=5$, $q=3$. The patterns are $'221'$ for $p$ and $'21'$ for $q$ (without $'12'$ because $p$ is odd). The largest $0-$major chord within the octave is $0\to 5\to 8\to 13$. The $0-$major scale is the sequence $0, 2, 4, 5, 7, 8, 10, 12, 13, 0$. The largest $0-$minor chord within the octave is $0\to 3\to 8\to 11$. Hence the $0-$minor scale is the sequence $0, 2, 3, 5, 7, 8, 10, 11, 0$. 
\end{example}

 \begin{example}
  For $\mathbb{Z}_{30}=\langle 5,  6\rangle$, $p=6$, $q=5$. The patterns are $'222'$ for $p$ and $'221'$/ $'122'$ for $q$. Note that the alternation in the $q$ pattern is used only in the minor scale. The largest $0-$major chord within the octave is  $0\to 6\to 11\to 17\to 22\to 28$. The $0-$major scale is the sequence $0,2,4,6, 8, 10, 11, 13$, $15, 17, 19, 21, 22, 24, 26, 28,  0$. The largest $0-$minor chord within the octave is $0\to 5\to 11\to 16\to 22\to 27 $. The $0-$minor scale is the sequence $0, 2, 4, 5, 7, 9, 11, 12, 14, 16, 18, 20, 22, 24, 26, 27, 0 $. Note that the legs of length $q=5$ are filled in by alternating $'221'$ with $'122'$.
\end{example}
\begin{remark}
In \cite{CM} scales are defined as sequences $x_i=[\frac{i\cdot k}{n}]$, $i=1,...,k$, where $k\in\mathbb{Z}_n$ is relatively prime to $n$. In our case $k=p+q$ is relatively prime to $n$, and we may define a similar scale. However, the major/minor alternatives will be lost because the major/minor chords on which the scales are built, are discarded by such a definition. 
\end{remark}
\begin{remark} In our context, triads such as $\{0,p, p+q \}$, $\{0,q, p+q \}$, $\{0,p, p+p \}$, $\{0,q, q+q \}$ can be used to construct hexatonics. Tonnetz spaces can be built using perpendicular axes for the generators $p$, $q$, and diagonal axes for the generalized circle of fifths and major/minor scales. For the $\mathbb{Z}_{12}$ case, see e.g. \cite{Hook}. 
\end{remark}

\noindent Aiming toward counterpoint, the two results below show that isometric reflections always exist in our set-up. 
\begin{proposition}\label{gwcp} 
 Let $n=pq$ with gcd$\{p,q\}=1$ and $\varphi:\mathbb{Z}_n\to \mathbb{Z}_n $ defined by $\varphi(x)=(n-1)\odot x$. Then $\varphi\in\text{ Aut }(\mathbb{Z}_n)$, $\varphi^2=Id$, and $\varphi$ is an isometry on the unoriented Cayley graph of $\mathbb{Z}_n$ with respect to any symmetric generating set $S$ of $\mathbb{Z}_n$. 
 \end{proposition}
 \begin{proof} Clearly gcd$\{n-1,n\}=1$, thus $\varphi\in\text{ Aut }(\mathbb{Z}_n)$. Notice that $\varphi$ is actually the reflection 
 $\varphi(x)=(nx-x)\text{ mod n}=\ominus x$, $\forall$ $x\in\mathbb{Z}_n$. Hence everything follows from Proposition \ref{ccpt} and Example \ref{refl}. 
 \end{proof}  
\begin{corollary}\label{wiso}
In the hypotheses of Proposition \ref{gwcp}, with arbitrary $w\in\mathbb{Z}_n$, the affine transformation $T:\mathbb{Z}_n\to \mathbb{Z}_n$ defined by $T(x)=\varphi(x)\oplus w$ is an isometry on the corresponding unoriented Cayley graph, and $T^2=Id$.     
\end{corollary}
\noindent We use the above corollary to show that the weak counterpoint condition from Definition \ref{condition} is available under mild restrictions.  
\begin{theorem}\label{last}   
 In the hypotheses of Proposition \ref{gwcp} and Corollary \ref{wiso}, with generating set $S=\{p,q, n-p, n-q\}$ and $K'=\{0\}\cup S$, the affine transformation $T(x)=(n-1)x\oplus w$ satisfies (\ref{wcpt}), for any $w\notin K'\oplus K'$.    
\end{theorem}
 \begin{proof}
Due to the above proposition and its corollary, we have that $T$ is an isometry and $T^2=Id$. To achieve the last requirement in (\ref{wcpt}) we check the contrapositive, i.e. if $T(K')\cap K'\neq\emptyset$ then $w\in K'\oplus K'$. Indeed, if $x\in T(K')\cap K'$ then $\exists$ $y\in K'$ such that $T(y)=(n-y)\oplus w=x\in K'$, hence $w=y\oplus x\in K'\oplus K'$.
 \end{proof}   
\noindent This last theorem tells us that a strict inclusion $K'\oplus K'\subsetneq \mathbb{Z}_n$ is sufficient to obtain the weaker form of counterpoint (\ref{wcpt}). 
If $n$ is even, one can push (\ref{wcpt}) to (\ref{cpt}), see proof of Theorem \ref{4p} where $n=12$, and example below where $n=10$. If $n$ is odd, one cannot push (\ref{wcpt}) to (\ref{cpt}), however, one can enlarge $K'$ to a largest possible set of consonants $K$, disjoint from $T(K)=D$ such that $K\cup D$ covers all but one element of $\mathbb{Z}_n$, see example below with $n=15$. We will consider this statement for arbitrary $n$ in a future work.   
\begin{example}
For $\mathbb{Z}_{10}=\langle 2,5,8\rangle$ we have $K'=\{0,2,5,8 \}$ and $K'\oplus K'=\{ 0,2,3,4,5,6,7,8\}$. Hence, with $w\in\{1,9\}$ the affine transformation $T(x)=9x\oplus w$ is an isometry which satisfies $T^2=Id$ and $T(K')\cap K'=\emptyset$, by Theorem \ref{last}. Let us also note that the only non-trivial automorphism $\varphi$ of $\mathbb{Z}_{10}$ which satisfies $\varphi^2=Id$ is $\varphi(x)=9\odot x$ (check this using  $U(10)$). Now, for the choice $w=1$, using the same ideas as in Theorem \ref{4p}, one finds two partitions $(K_{1,2}, D_{1,2})$ of $\mathbb{Z}_{10}$ associated to affine $T(x)=9x\oplus 1$ satisfying $T(K_i)=D_i$, i.e. the counterpoint requirements in (\ref{cpt}) are met. These partitions are given by $K_1=\{0,2,5,8,4\}$ and $K_2=\{0, 2,5,8,7\}$. The choice $w=9$ gives two more counterpoint partitions, $(K_{3,4}, D_{3,4})$ of $\mathbb{Z}_{10}$ with respect to $T(x)=9x\oplus 9$. These are given by $K_3=\{0,2,5,8,6\}$ and $K_4=\{0, 2,5,8,3\}$.
\end{example}
\begin{example}
For $\mathbb{Z}_{15}=\langle 3,5,10,12 \rangle$ we have $K'=\{0,3,5,10,12 \}$ and \\
$K'\oplus K'=\{ 0, 3, 5, 10, 12, 6, 8, 13, 2, 7,9 \}$. Hence, for any $w\in\{1,4,11,14 \}$ the affine transformation $T(x)=14x\oplus w$ is an isometry which satisfies $T^2=Id$ and $T(K')\cap K'=\emptyset$ by Theorem \ref{last}. Because $15$ is odd, it is not possible to obtain the counterpoint condition precisely, but in each valid case for $w$ we can ``push" $K'$ toward a bigger set of consonants $K\subset\mathbb{Z}_{12}$. Note in this case one can increase $K'$ with no more than two intervals because $|K'|=5$, and we want $K'\subset K$, $T(K)\cap K=\emptyset$ with bijective $T$; thus $|K|+|T(K)| = 2|K| < 15$. We set out to find two more values $z_1\neq z_2\in \mathbb{Z}_{15}\setminus K'$ such that their images $T(z_1)\neq T(z_2)$ belong to $\mathbb{Z}_{15}\setminus K'\cup \{ z_1, z_2\}$. With $w=1$ for example, $T(x)=14x\oplus 1$, and we are looking to find sets $K:=K'\cup \{z_1, z_2\}$ and $D:=T(K)=\{1,13,11,6,4\}\cup\{T(z_1), T(z_2)\}$ such that $K\cap D=\emptyset$. Hence we need find which $z_1\neq z_2\in\{ 2,7,8,9,14\}$ meet the requirement. Because $T(8)=8$ we eliminate $z_{1,2}=8$. If $z_1=2$ then $T(z_1)=14$ and any of the cases $z_2\in\{7,9\}$ will do. For example with $z_2=7$, $T(z_2)=9$ and the disjoint sets  $K=\{0,3,5,10,12,2,7\}$ and $D=\{1,13,11,6,4,14, 9\}$ act as a consonant/dissonant pair in $\mathbb{Z}_{15}$, extending the weak counterpoint condition. Let us  note that there are two non-trivial automorphisms $\varphi$ of $\mathbb{Z}_{15}$, other than $\varphi(x)=14\odot x$, which satisfy $\varphi^2=Id$, namely $\varphi_1(x)=4\odot x$ and $\varphi_2(x)=11\odot x$. One can check that condition i) in Theorem \ref{iso} is satisfied to conclude that both $\varphi_{1,2}$ are isometries. Hence, these can also be used to test for what $w\in\mathbb{Z}_{15}$ the affine map   $T=\varphi_{1,2} \oplus w$ achieves $T(K')\cap K'=\emptyset$, i.e. the weak counterpoint condition. 
\end{example}

\section{Conclusion}
\noindent We have interpreted various concepts from music theory through the lens of the Cayley (un)oriented graphs associated to the group $\mathbb{Z}_{12}$. Using the Cayley graph as a guiding principle, we have defined and studied chords, scales, circle of fifths, and first species counterpoint partitions in the setting of a group $\mathbb{Z}_n$ generated by two relatively prime numbers $p$ and $q$ such that $n=pq$. We have written Maple code to implement  and practically experience these concepts.

\begin{acknowledgements} 
We would like to thank Catalin Georgescu, Connor Gibbs, Amy Laursen, and Paul Lombardi for very useful conversations and suggestions which greatly improved our paper. The first named author also thanks Dan Van Peursem for crucial help with building two cigar box guitars one of which is tuned to the notes and chords of the musical system  $(\mathbb{Z}_{15}, 2)$. 
This work was partially supported by a UDiscover grant from the University of South Dakota (to Olivia Roberts).

\end{acknowledgements}


\begin{thebibliography}{999}

\bibitem{AJM}
O.A. Agustin-Aquino, J. Junod, G. Mazzola, \emph{Computational Counterpoint Worlds}, Computational Music Science, Springer (2015)


\bibitem{AAL}
J.S. Arias-Valero, O.A. Agustín-Aquino and E. Lluis-Puebla, \emph{On First-Species Counterpoint Theory}, Brazilian Journal of Music and Mathematics, Vol. V, No. 2, pp. 1-40, (2021)


\bibitem{AKK}
A. AsKew, K. Kennedy and V.W. Klima (2018),\emph{ Modular Arithmetic and Microtonal Music Theory},  PRIMUS, vol. 8, issue 2, pp 458-471 (2018) 

\bibitem{Baz}
G. Balzano, \emph{The group-theoretic description of 12-fold and microtonal
pitch systems}, Computer Music Journal. 4(4): 66–84 (1980)

\bibitem{Bo}
H. Bohlen,  (1978). "13 Tonstufen in der Duodezime". Acoustica (in German). Stuttgart: S. Hirzel Verlag. 39 (2): 76–86 

\bibitem{CM}
J. Clough and G. Myerson, \emph{Musical scales and the generalized circle of fifths}, Amer.Math.Monthly, 93/9:695-701 (1986)

\bibitem{Fux}
J.J. Fux, \emph{Gradus ad Parnassum (1725), The Study of Counterpoint}. Translated and edited by Alfred Mann, Norton (1971)


\bibitem{Ga}
  J.A. Gallian,  \emph{Contemporary Abstract Algebra}, $8^{\text{th}}$ Edition,Brooks/Cole, Cengage Learning, (2010)
  
\bibitem{Gan}
K. Gann, 
\emph{The Arithmetic of Listening}, University of Illinois Press, (2019)

\bibitem{Gou}
M. Gould,\emph{ Balzano and Zweifel: Another look at generalized diatonic
scales}, Perspectives of New Music, 38(2): 88–105 (2000)

\bibitem{Hook}
J. Hook, \emph{Exploring Musical Spaces}, Oxford University Press, 2023

\bibitem{TJ}
  T.W. Judson, \emph{Abstract Algebra: Theory and Applications},
  Orthogonal Publishing, (2020)


\bibitem{Cai}
C.H. Li, \emph{On isomorphism of finite Cayley graphs - a survey }, Discrete Math., 256(1-2), pp.301-334, (2002)

\bibitem{Loh}
C. L\"{o}h, \emph{Geometric Group Theory, An Introduction}, Universitext, Springer (2017)

\bibitem{Ma}
G. Mazzola, Y. Pang, W. Heinze, K. Gkoudina, G. Afrisando Pujakusuma, J. Grunklee, Z. Chen, T. Hu, Y. Ma,\emph{ Basic Music Technology. An Introduction}, Springer (2018) 

\bibitem{Mco}
G. Mazzola,\emph{The Topos of Music, I: Theory}, Second Edition, Springer (2017)

\bibitem{Pi}
G. Picioroaga, \emph{Understanding math concepts in music}. In Corless, Gerhard, Kotsireas (ed.). Maple in Mathematics Education and Research, 2020. Communications in Computer and Information Science, Springer 

\bibitem{Pie}
J. Pierce, \emph{Consonance and scales}. In Cook, Perry (ed.). Music, Cognition, and Computerized Sound: An Introduction to Psychoacoustics, MIT Press, (2001)


\bibitem{Rus}
T.W. Rush, \emph{Music Theory for Musicians and Normal People}, https://tobyrush.com

\bibitem{Ty}
D. Tymoczko, \emph{The geometry of musical chords}, Science 313, 72 (2006)

\bibitem{Vi}
A. Vieru. \emph{ Cartea modurilor, 1 (Le livre des modes, 1}). Ed.Muzicala, Bucarest, 1980. Revised ed. The book of modes, 1993


\end{thebibliography}
\end{document}